\documentclass{siamart171218} 
\usepackage{chngcntr}
\usepackage{amsopn}
\usepackage{amsmath}
\usepackage{amssymb}
\usepackage{graphicx}
\usepackage{capt-of}
\usepackage{multirow}
\usepackage{makecell}
\usepackage{multicol}
\usepackage{float}
\usepackage{todonotes}
\usepackage{chngcntr}
\usepackage{multicol}
\usepackage{cleveref}
\usepackage{url}

\usepackage{tikz}
\usetikzlibrary{calc}
\usetikzlibrary{shapes, patterns, arrows}
\usepackage{pgfplots}
\usepackage{pgfplotstable}
\pgfplotsset{compat=1.14}
\usepackage{standalone}

\usepackage{varwidth}
\newsavebox\tmpbox

\usepackage{thm-restate}
\usepackage{algpseudocode}  
\usepackage[utf8]{inputenc}

\newcommand{\dual}{'}
\newcommand{\half}{{1/2}}
\newcommand{\reals}{{\mathbb{R}}}

\newcommand{\semi}[1]{\lvert#1\rvert}

\newcommand{\eps}{\varepsilon}
\newcommand{\pelm}[1]{\text{P}\textsubscript{#1}}

\newcommand{\jump}[1]{\ensuremath{[\![#1]\!]} }
\newcommand{\avg}[1]{\ensuremath{\left\{\!\left\{#1\right\}\!\right\}} }

\newcommand{\uf}{\ensuremath{\mathbf{u}_f}}
\newcommand{\up}{\ensuremath{\mathbf{u}_p}}
\newcommand{\ff}{\ensuremath{\mathbf{f}_f}}
\newcommand{\fp}{\ensuremath{f_p}}
\newcommand{\hf}{\ensuremath{\mathbf{h}_f}}
\newcommand{\hp}{\ensuremath{h_p}}
\newcommand{\nablab}{\nabla}

\newcommand{\vf}{\ensuremath{\mathbf{v}_f}}
\newcommand{\vp}{\ensuremath{\mathbf{v}_p}}
\newcommand{\nG}{\ensuremath{\mathbf{n}}}
\newcommand{\tG}{\ensuremath{{T}_t}}
\newcommand{\la}[1]{\ensuremath{\mathbf{#1}}}

\newcommand{\BJS}{\ensuremath{\alpha_{\text{BJS}}}}
\newcommand{\Epsilon}{\boldsymbol{\epsilon}}
\newcommand{\bsigma}{\boldsymbol{\sigma}}

\newcommand{\Pelm}[1]{\textit{P}\textsubscript{#1}}
\newcommand{\RTelm}[1]{\textit{\textbf{RT}}\textsubscript{#1}}
\newcommand{\DSelm}{\textit{\textbf{P}}\textsubscript{2}-\Pelm{1}-\RTelm{0}-\Pelm{0}-\Pelm{0}}
\newcommand{\SNelm}{\textit{\textbf{P}}\textsubscript{2}-\Pelm{1}-\textit{\textbf{P}}\textsubscript{2}-\Pelm{1}-\textit{\textbf{P}}\textsubscript{0}}
\newcommand{\Selm}{\textit{\textbf{P}}\textsubscript{2}-\Pelm{1}-\Pelm{0}}
\newcommand{\vecSelm}{\textit{\textbf{P}}\textsubscript{2}-\Pelm{1}-\textit{\textbf{P}}\textsubscript{0}}

\newcommand{\Delm}{\RTelm{0}-\Pelm{0}-\Pelm{0}}

\newtheorem{remark}{Remark}[section]
\newtheorem{assumption}{Assumption}[section]

\newtheorem{example}{Example}[section]

\renewcommand{\AA}{\mathcal{A}}
\newcommand{\BB}{\mathcal{B}}

\newcommand{\set}[1]{\{#1\}}

\usepackage{chngcntr}
\counterwithin{table}{section}
\newcommand{\Rsz}[1]{\ensuremath{R_{#1}}}
\newcommand{\Rszi}[1]{\ensuremath{R^{-1}_{#1}}}

\newcommand{\Hdiv}[1]{\ensuremath{\mathbf{H}(\operatorname{div},\,#1)}}
\newcommand{\HdivD}[2]{\ensuremath{\mathbf{H}_{0, #1}(\operatorname{div},\,#2)}}
\newcommand{\ismuL}{{\tfrac 1 {\sqrt{\mu}}L^2(\Omega_f)}}

\newcommand{\smuH}{{\sqrt{\mu}\textbf{H}^1_{0, D}(\Omega_f) \cap \sqrt{D}\textbf{L}^2_{t}(\Gamma)}}
\newcommand{\isKHd}{\tfrac 1 {\sqrt{K}} \HdivD{D}{\Omega_p}}
\newcommand{\sKL}{{\sqrt{K}} L^2(\Omega_p)}

\newcommand{\DSmultf}{\tfrac 1 {\sqrt{\mu}} H^{-1/2}(\Gamma)}

\newcommand{\DSmultpD}{\sqrt{K} H_{00}^{1/2} (\Gamma)}

\title{Robust preconditioning of monolithically coupled multiphysics problems
}

\headers{Monolithic multiphysics preconditioners}{K. E. Holter, M. Kuchta, and K. A. Mardal}

\author{Karl Erik Holter \thanks{
    Simula Research Laboratory, Fornebu, Norway (\email{karl0erik@gmail.com}, \email{miroslav@simula.no}).\funding{Karl Erik Holter is funded by the Simula-UCSD-University of Oslo Research and PhD training (SUURPh) program, an international collaboration in computational biology and medicine funded by the Norwegian Ministry of Education and Research.
}}
  \and Miroslav Kuchta\footnotemark[2]
\and Kent-Andre Mardal\footnotemark[2]$^{\;,}$\thanks{
  Department of Mathematics, University of Oslo, Norway (\email{kent-and@simula.no}).
  }}

\usepackage{amsopn}
\DeclareMathOperator{\diag}{diag}

\begin{document}

\maketitle 
\begin{abstract}
{In many applications, one wants to model physical systems consisting of two 
different physical processes in two different domains that are coupled 
across a common interface. A crucial challenge is then that the solutions of the two 
different domains often depend critically on the interaction at the interface
and therefore the problem cannot be easily decoupled into its subproblems.
Here, we present a framework for finding
robust preconditioners for a fairly general class of such problems
by exploiting operators representing fractional and weighted Laplacians at the interface. 
Furthermore, we show feasibility of the framework for two common multiphysics problems; namely 
the Darcy-Stokes problem and a fluid--structure interaction problem. Numerical 
experiments that demonstrate the effectiveness of the approach are included.
}
\end{abstract}



\section{Introduction}\label{sec:intro}
This paper is concerned with preconditioning of monolithic schemes for multiphysics 
problems where two single-physics problems are coupled at a common interface. 
We will employ operator preconditioning and fractional and weighted Sobolev spaces
in order to establish preconditioners that are parameter robust and order
optimal with respect to the resolution of the mesh. 
{Two multiphysics problems will be considered: a Darcy-Stokes problem coupling viscous flow
  to porous media flow and a fluid--structure interaction (FSI) problem involving viscous fluid flow
  and small linear deformations of the solid.}

We shall illustrate the concepts using the Darcy-Stokes problem. Let $\Omega_f$ and $\Omega_p$ be the domain of the viscous flow
and the porous medium, respectively, and $\Gamma=\partial\Omega_f\cap\partial\Omega_p$ be
their common non-empty interface. Further let the subdomains' boundaries be decomposed as $\partial \Omega_i = \Gamma \cup \partial \Omega_{i,D} \cup \partial \Omega_{i, N}$,
$i=f, p$, see Figure \ref{fig:DSdomains} for illustration. Here the subscripts $D, N$ signify respectively that Dirichlet and Neumann boundary conditions are prescribed on the part of the boundary. Our interest concerns applications where boundary conditions typically 
include both Dirichlet and Neumann conditions. We will therefore pay special attention to the boundary conditions and consider
cases in which, to the authors' knowledge, the well-posedness of the subproblems has not been established theoretically.  
In these cases we will pose assumptions on the subproblems which imply well-posedness of the coupled problem and include numerical experiments
that support these assumptions.

The 2D coupled Darcy-Stokes problem reads\\
\begin{minipage}{0.52\textwidth}
  \begin{subequations}\label{eq:darcy_stokes_strong}
    \begin{align}
      \label{eq:darcystokes1} 
      -\mu \Delta \uf + \nabla p_f &= \ff  &\text{ in } \Omega_f,\\
      \label{eq:darcystokes2}
      \nabla \cdot \uf &= 0 &\text{ in } \Omega_f,  \\
      \label{eq:darcystokes3} 
      K^{-1} \up + \nabla p_p &= 0  &\text{ in } \Omega_p,\\
      \label{eq:darcystokes4}
      \nabla \cdot \up &= \fp &\text{ in } \Omega_p,  \\
      \label{eq:darcystokesmass} 
      \up \cdot \nG - \uf \cdot \nG &= 0 &\text{ on } \Gamma, \\
      \label{eq:darcystokesstress}
      - \mu \frac {\partial \uf} {\partial \nG} \cdot \nG + p_f &= p_p  &\text{ on } \Gamma ,\\
      \label{eq:darcystokesBJS}
  - \mu \frac {\partial \uf} {\partial \nG} \cdot \boldsymbol{\tau} - D \uf \cdot \boldsymbol{\tau} &= 0  &\text{ on } \Gamma.
\end{align}
\end{subequations}
\null
\par\xdef\tpd{\the\prevdepth}
\end{minipage}
\begin{minipage}{0.47\textwidth}
  \begin{figure}[H]
    \centering
    \includegraphics[width=\textwidth]{img/inline_tex_DSdomain.tex}
    \vspace{-25pt}
    \caption{Schematic domain of Darcy-Stokes problem.
    }
    \label{fig:DSdomains}
\end{figure}
\end{minipage}


\vspace{0.2cm}
Here, $\uf, p_f$ are the unknown velocity and pressure for the Stokes problem \eqref{eq:darcystokes1}-\eqref{eq:darcystokes2}
in $\Omega_f$ and $\up, p_p$ are the unknown velocity and pressure of the Darcy problem \eqref{eq:darcystokes3}-\eqref{eq:darcystokes4}
in $\Omega_p$. We remark that below we will change the sign of the pressures in order to get a symmetric problem. 
The (constant) material parameters are the fluid viscosity $\mu$, the hydraulic conductivity $K$ and
$D = \alpha_{\text{BJS}}\sqrt{\frac{\mu }{K}}$ with $\alpha_{\text{BJS}}$ the Beavers-Joseph-Saffman (BJS) coefficient. Finally,
$\nG$ is the unit outer normal of the subdomains (on $\Gamma$ the normal is oriented with respect
to $\Omega_f$) and $\boldsymbol{\tau}$ is a unit vector tangent to the interface. At the interface
$\Gamma$ the conditions
\eqref{eq:darcystokesmass}--\eqref{eq:darcystokesBJS} are respectively conservation of mass,
balance of normal stress and the BJS condition \cite{mikelic2000interface}. We further assume that the
problem is equipped with the following boundary conditions
\[
\uf = \uf^0\mbox{ on }\partial\Omega_{f, D},\quad
\up\cdot\boldsymbol{n} = u^0_p\mbox{ on }\partial\Omega_{p, D}
\]
and
\[
\mu\frac{\partial\uf}{\partial\nG} - p\nG = \hf \mbox{ on }\partial\Omega_{f, N},\quad
p_p = \hp \mbox{ on }\partial\Omega_{p, N}.
\]

The well-posedness of the coupled problem \eqref{eq:darcystokes1}-\eqref{eq:darcystokesBJS} is well-known
in the case of Dirichlet conditions at the boundary, c.f.~\cite{layton2002coupling, galvis2007non}. 
Our work here is related to~\cite{galvis2007non} where error estimates that were 
robust with respect to variations in the material parameters were obtained 
in formulations using a Lagrange multiplier at the interface. 
Robust preconditioners were, however, not discussed in either~\cite{layton2002coupling, galvis2007non}.
Still, 
the precise tracking of parameters  in the norms in~\cite{galvis2007non}
serves as an excellent starting point for deriving robust preconditioning.  
Here, we will show that the norms of the fluid velocity and pressure 
in both the viscous and porous domains can be derived from their analysis, but 
that there are important differences for the Lagrange multiplier at the interface.

Our main motivation for the current study of multiphysics systems are the viscous-porous-elastic coupled problems in a biomechanical
setting. Here, the material properties do not vary significantly themselves, e.g.  
the viscosity of blood is typically around 3\mbox{mPa$\cdot$s}, while 
water has viscosity around 0.7\mbox{mPa$\cdot$s}, which is also a good approximation for cerebrospinal fluid, plasma and extracellular fluid.   
Furthermore, the permeability in tissue is typically in the order of 
$10^{-15}\mbox{m}^2$ to $10^{-18}\mbox{m}^2$~\cite{holter2017interstitial,koch2018multi,sarntinoranont2006computational,smith2007interstitial,stoverud2016poro}.
The physical parameters thus do not vary significantly. However,
the length scales span from $\mbox{dm}$ to $\mu\mbox{m}$ and introduce
variations that require parameter robustness.
For example, permeability alone, which has units of length squared, introduces parameter
changes of order $10^{10}$ in viscous-porous coupling from the macro-circulation level 
at $\mbox{dm}$-scale~\cite{sarntinoranont2006computational,smith2007interstitial, stoverud2016poro}
to the micro-circulation  at the $\mu\mbox{m}$-scale~\cite{holter2017interstitial,koch2018multi}.          


Discretization of coupled multiphysics problems is challenging because 
the subproblems may require different approaches. For example, in \eqref{eq:darcy_stokes_strong} 
the $\la{H}(\text{div})$ conforming elements typically used for the Darcy flux, 
e.g. Raviart-Thomas element, do not provide stable discretization of the Stokes velocity. 
From an implementation point of view, it may be beneficial to employ the
same discretization in both domains, so-called unified approaches, and several strategies have
been proposed~\cite{badia2009unified, burman2007unified,karper2009unified,rui2009unified}. 
Alternatively, in the non-unified approach the  discretizations best suited for the 
subproblems are used. However, then a proper coupling of the schemes across the 
interface presents a challenge. For the coupled Darcy-Stokes problem \eqref{eq:darcy_stokes_strong} 
such stable element pairs are given e.g. in \cite{galvis2007non, gatica2008conforming, layton2002coupling, riviere2005locally}. 
Here, we shall further use the discretization proposed by \cite{galvis2007non}.

The solution approaches for coupled multiphysics problems can in general
be divided into monolithic solvers (where all the problem unknowns are solved
for at once) or domain-decomposition (DD) solvers (where one iteratates between
the sub-problems). For Darcy-Stokes problem these have been applied both
to the mixed form \eqref{eq:darcy_stokes_strong} and the primal form, in which 
the Darcy problem is only solved for the pressure and which results in a 
non-symmetric problem, see \cite{discacciati2002mathematical}.
  Monolithic multigrid solvers for the mixed formulation have been proposed in
  \cite{luo2017uzawa}, while balancing domain decomposition
preconditioner and the mortar formulation suitable for DD preconditioning
are discussed in \cite{galvis2007balancing} and \cite{girault2014mortar}
respectively. Concerning the (non-symmetric) primal formulation,
\cite{cai2009preconditioning, chidyagwai2016constraint} studied monolithic
solvers based on preconditioned GMRES. Domain decomposition algorithms based
on Dirichlet-Neumann, or Robin-Robin coupling are then discussed 
in \cite{discacciati2004convergence} or e.g. \cite{discacciati2007robin, chen2011parallel}.
Multigrid approaches were proposed in \cite{mu2007two, cai2012multilevel}.
We remark that of the cited works only \cite{luo2017uzawa, discacciati2007robin}
present algorithms which are robust in discretization and material parameters.

Multigrid preconditioners for the fluid-structure interaction problem solved
with GMRES are discussed e.g. in \cite{hron2006monolithic} (Vanka smoother) or
\cite{gee2011truly} (using Gauss-Seidel). Different block preconditioners
for GMRES are then discussed in \cite{heil2004efficient} while \cite{badia2009robin}
derive preconditioner based on DD and Robin-Robin coupling.
Domain decomposition solvers for the FSI based on interaction between the fluid
and the solid via Lagrange multipliers are proposed in \cite{gerstenberger2008extended}.
Finally, \cite{deparis2006fluid, deparis2006domain, badia2008fluid} derive FSI
solvers considering preconditioned Richardson iterations for the related (interfactial) Steklov-Poincar{\' e}
operators. We remark that the problem to be studied in \S \ref{sec:stokes-elasticity}
shall be viewed as a component of an FSI solver, in particular, we consider
a fixed interface and a linear material.

To the authors' knowledge, order optimal monolithic preconditioning,
devised using the operator preconditioning framework, which is robust
with respect to any variations in material parameters has
not been accomplished for flow problems involving the coupling between viscous and porous flow and 
fluid--structure interaction problems. Our aim here is to 
devise such preconditioners. However, our analysis is restricted to coupled problems
where the dynamics is slow and linear. Hence, the viscous flow problem is in both of the coupled
problems represented by Stokes equations. Furthermore, the 
fluid--structure problem we consider here is  
the coupling of viscous fluid described by Stokes model and linear elastic solid described by
Navier's elasticity equation. We will hence refer to the fluid--structure problem 
as a Stokes-Navier problem to distinguish it from the common Navier-Stokes equations 
of fluid flow as well as fluid--structure problems in general.  
Our goal here is therefore to describe parameter robust preconditioners for
both the Darcy-Stokes and Stokes-Navier problems.   
To this end we shall crucially rely on operators in fractional
Sobolev spaces. More specifically, by considering the monolithic saddle point problem
consisting of both subproblems coupled together with a Lagrange multiplier, we shall
establish a formulation with the Lagrange multiplier in properly weighted fractional spaces such that we  can derive 
parameter-robust stability estimates and corresponding preconditioners.

An outline of the paper is as follows: Section \ref{sec:prelims} describes notation and the
mathematical setting in which we operate. In Section \ref{sec:abstract} we present the framework
for deriving preconditioners for (a class of) coupled multiphysics problems. The
framework is then applied to derive robust preconditioners for the Darcy-Stokes system in Section \ref{sec:darcy-stokes} 
and for the Stokes-Navier system in Section \ref{sec:stokes-elasticity}.

\section{Preliminaries}\label{sec:prelims}
We will use boldface symbols to denote vector fields and spaces of vector fields 
while scalar fields and spaces are written in a normal font. Similar distinction 
will not be made for the operators as their meaning shall always be clear from the 
context.

Let $\Omega$ be a bounded Lipschitz domain in $\reals^n,\, n=2,3$
and $L^2=L^2(\Omega)$ be the Lebesgue space of square integrable functions. Sobolev
spaces with derivatives of order up to $k$ in $L^2$ are denoted
by $H^k$ whereas $H^k_0$ denotes the closure of $C_0^\infty(\Omega)$ in $H^k$.
The Sobolev space of $\la{L}^2$ functions whose divergence is in $L^2$ is denoted $\Hdiv{\Omega}$. 
{Function spaces containing only functions with mean value zero are denoted
as quotient spaces, e.g. $L^2(\Omega)/\reals$ is the space of $L^2$ functions 
on $\Omega$ with mean value zero}.

The dual space of a vector space $X$ is denoted as $X'$. For two normed vector spaces $X, Y$ the space of bounded
linear operators mapping $X$ to $Y$ is denoted $\mathcal{L}(X, Y)$, or just $\mathcal{L}(X)$ if $Y=X$. The
inner product on a space $X$ is denoted $(\cdot, \cdot)_X$. For simplicity, the $L^2$-inner product between scalar, vector and tensor fields in $L^2$ as well
as the duality pairing between a Hilbert space and its dual is denoted by $(\cdot, \cdot)$.
We shall sometimes (in the interest of clarity) indicate the domain in the $L^2$ inner product by a subscript,
e.g. $(\cdot, \cdot)_{\Gamma}$. The dual of an 
operator $B$ with respect to the $L^2$ inner product is denoted by $B\dual$. 
The Riesz mapping of a Hilbert space $V$ is denoted as $\Rsz{V}$ and $\Rsz{V}:V\dual\rightarrow V$.
Its inverse map is denoted as $\Rszi{V}: V \to V\dual$.

If $X, Y$ are Sobolev spaces, and $a$ an arbitrary {positive} real number,
we define the weighted space $aX$ to be the space $X$ with the norm $a\|\cdot\|_X$.
The intersection $X \cap Y$ and sum $X+Y$ are Hilbert spaces with norms 
\begin{eqnarray*}
\|u\|_{X \cap Y} = \sqrt{\|u\|^2_X + \|u\|^2_Y} \quad\quad\mbox{and}\quad\quad
\|u\|_{X + Y} = \inf_{\substack{x + y = u \\x \in X, y\in Y }} \sqrt{\|x\|^2_X + \|y\|^2_Y} .
\end{eqnarray*}


Following \cite{kuchta2016preconditioners}, we define the Sobolev space $H^s(\Omega)$
for a real number $s \in (-1, 1)$  in terms of the spectral decomposition of Laplacian.
This definition is easily implementable and suitable for our purposes, but numerous
alternative definitions exist, whose equivalence to the spectral definition used here
depends on boundary conditions. We will not go into detail here, but refer to \cite{di2012hitchhikers,lischke2018fractional} for an
overview, and to \cite{chandler2015interpolation} for a treatment of our definition in terms
of interpolation spaces. 
Let $S \in \mathcal{L}(H^1(\Omega))$ be the operator such that $(Su, v)_{H^1(\Omega)} = (u, v)_{L^2(\Omega)}$
for all $v \in H^1(\Omega)$, where we remark that we use the full $H^1$ norm. 
We can then find a basis $\{\phi_i\}_i$ of eigenvectors of $S$ for
$H^1(\Omega)$ with eigenvalues $\lambda_i>0$, and for any $u = \sum_i c_i u_i$ define
$$\|u\|_{H^s} = \sqrt{\sum_i c_i^2 \lambda_i^{-s}}.$$ $H^s$ is then the closure of
$\text{span } \{\phi_i\}_i$ in $\|\cdot\|_{H^s}$.

{
For a Lipschitz domain $\Omega$ with $\Gamma \subseteq \partial \Omega$
we define a trace operator $T$ and normal trace operator $T_n$ such that
\[
(Tu)(x) = u(x), x\in\Gamma, \forall u\in C^{\infty}(\overline{\Omega})
\]
and
\[
(T_n\boldsymbol{u})(x) = \boldsymbol{u}(x)\cdot\nG(x), x\in\Gamma, \forall \boldsymbol{u}\in (C^{\infty}(\overline{\Omega}))^d.
\]
The trace operator acting on vector fields is likewise denoted $T$ and is defined
component wise. The tangential trace operator $\tG$ is defined analogously to $T_n$ and we
let $\la{L}_{t}^2(\Gamma)$ be the space of functions $\boldsymbol{u}$ on $\Omega$ such
that $\boldsymbol{u}\cdot\boldsymbol{\tau}\in L^2(\Gamma)$.

Following~\cite{galvis2007non}, for 
$\Gamma$ a subset of $\partial\Omega$, we define $H^{\half}_{00}(\Gamma)$ to be the space of all
$w \in H^{\half}(\Gamma)$ for which the extension by 0 to $\partial\Omega$ is in
$H^{\half}(\partial\Omega)$.  We also define $H^{-\half}_{00}(\Gamma)$ to be the dual of
$H^{\half}_{00}(\Gamma)$, and denote the extension by zero as $E_{00}: H_{00}^{1/2}(\Gamma) \to H^{1/2}(\partial\Omega)$.
With these definitions the trace operators can be extended to surjective and continuous mappings with a
bounded right inverse; 
$T: H^1(\Omega)\rightarrow H^{1/2}(\partial \Omega)$, see \cite[Thm. 3.37]{mclean2000strongly}
and 
$T_n: \Hdiv{\Omega}\rightarrow H^{-1/2}(\partial \Omega)$ 
see \cite[Thm. 2.5]{girault2012finite} and \cite[Cor. 2.8]{girault2012finite}.
}

As the restriction map $\rvert_{\Gamma}: H^{\half}(\partial \Omega) \to H^{\half} (\Gamma)$ is well-defined
for any $\Gamma \subset \partial \Omega$ we can define $T: H^1(\Omega)\rightarrow H^{1/2}(\Gamma)$
by composition. Taking the kernel of this map, we define the space $H^1_{0, \Gamma}(\Omega)$ of $H^1$ functions
whose restriction to $\Gamma$ is zero.
However, as the restriction $\rvert_{\Gamma}$ in $H^{-\half}$  is in general not surjective, we cannot define a
similar restriction $\rvert_{\Gamma}: H^{-\half}(\partial \Omega) \to H^{-\half} (\Gamma)$. To define
the corresponding space $\HdivD{\Gamma}{\Omega}$, we therefore require a notion of what $w\rvert_{\Gamma} = 0$
means for $H^{-\half}(\partial \Omega)$. Following \cite{galvis2007non}, we say that $w\rvert_{\Gamma} = 0$
for $w \in H^{-\half}(\partial \Omega)$ if $(w, E_{00}v) = 0$ for all $v \in H_{00}^{1/2}(\Gamma)$, and
define the space $\HdivD{\Gamma}{\Omega}$ to be the space of all $\mathbf{u} \in \Hdiv{\Omega}$ for
which $\left(T_n\mathbf{u}\right)\rvert_{\Gamma} = 0$.

To avoid proliferation of subscripts, when $\partial \Omega_D \subset \partial \Omega$ and there is no possibility of confusion, we will denote the space $H^1_{0, \partial \Omega_D}(\Omega)$ of $H^1$ functions with homogeneous Dirichlet conditions at $\partial \Omega_D$ by $H^1_{0, D}(\Omega)$. Similarly, the space $\HdivD{\partial \Omega_ D}{\Omega}$ is denoted $\HdivD{D}{\Omega}$. Here $\partial \Omega_D$ may be any non-empty subset of the boundary, including the entire boundary.

\begin{remark} \label{rmk:tracedirichletbdy}
  From Lemma 2.2 of \cite{galvis2007non}, $T_n$ can be viewed as 
  mapping $\Hdiv{\Omega}$ to $H^{-1/2}(\partial \Omega \backslash \Gamma)$,
  defining the term $(T_n \mathbf{u}, w)_{\Gamma}$ for all $\mathbf{u} \in \HdivD{\partial {\Omega} \backslash \Gamma}{\Omega}$,
  $w \in H^{\half}(\Gamma)$. In general, if $\Gamma, \partial \Omega_D \subset \partial \Omega$ are such that $\Gamma \subset \partial \left ( \partial \Omega_D \right )$ then $T_n\mathbf{u}$ will be an
  element of $H^{-\half}(\Gamma)$ for all $\mathbf{u} \in \HdivD{D}{\Omega}$. If $\partial \Gamma \not \subset \partial \left ( \partial \Omega_D \right )$, this is no longer the case, although by defining
  $(T_n \mathbf{u}, w)_{\Gamma} := (T_n\mathbf{u}, E_{00}w)$ for all $w \in H_{00}^{1/2}(\Gamma)$, $T_n \mathbf{u}$ can be seen to lie in
  the space $H_{00}^{-\half}(\Gamma)$. Thus $T_n: \Hdiv{\Omega} \rightarrow H^{-1/2}(\partial \Omega)$ maps $\HdivD{D}{\Omega}$ to $H^{-1/2}(\Gamma)$ when $\partial \Gamma \subset \partial \Omega_D$, and to $H_{00}^{-1/2}(\Gamma)$ when $\partial \Gamma \not \subset \partial \Omega_D$. Similarly, by definition the trace operator $T: H^1(\Omega)\rightarrow H^{1/2}(\partial \Omega)$ maps the space $H^1_{0, D}(\Omega)$ to $H^{1/2}_{00}(\Gamma)$ when $\partial \Gamma \subset \partial \Omega_D$, and to $H^{1/2}(\Gamma)$ when $\partial \Gamma \not \subset \partial \Omega_D$.

\end{remark}

We will in this paper employ the operator preconditioning framework, see~\cite{mardal2011preconditioning}
for an overview. Hence, we briefly review the theory.  Let $\AA_\eps:X_\eps \rightarrow X_{\eps}\dual$ be an invertible symmetric isomorphism such
that
\begin{equation}
\label{abstractiso}
\|\AA_\eps\|_{\mathcal{L}(X_\eps, X_\eps\dual)} \le C_1 
\mbox{ and }
\|\AA^{-1}_\eps\|_{\mathcal{L}(X_\eps\dual, X_\eps)} \le C_2,  
\end{equation}
where the constants, $C_1$ and $C_2$,  are independent of the parameter $\eps$,  and $\eps$ 
may be a collection of parameters such as viscosity, permeability, the Lam{\' e} parameters
and the Beavers-Joseph-Saffman parameter. A parameter robust preconditioner is then derived as 
a Riesz mapping $\BB_\eps$ or an operator which is
spectrally equivalent with the Riesz mapping such that 
\[
\|\BB_\eps\|_{\mathcal{L}(X_\eps\dual, X_\eps)} \le C_3 
\mbox{ and }
\|\BB^{-1}_\eps\|_{\mathcal{L}(X_\eps, X_\eps\dual)} \le C_4.
\]
Here $C_3=C_4=1$ for the Riesz map, but in general we only require that the constants are bounded independently of the parameters.  
By construction, 
\[
\|\BB_\eps \AA_\eps \|_{\mathcal{L}(X_\eps, X_\eps)} \le C_1 C_3 
\mbox{ and }
\|(\BB_\eps \AA_\eps)^{-1} \|_{\mathcal{L}(X_\eps, X_\eps)} \le C_2 C_4
\]
and hence the condition number will be bounded  
\[
\operatorname{cond}( \BB_\eps \AA_\eps) = \|\BB_\eps \AA_\eps \|_{\mathcal{L}(X_\eps, X_\eps)} 
\|(\BB_\eps \AA_\eps)^{-1} \|_{\mathcal{L}(X_\eps, X_\eps)} 
\le C_1 C_2 C_3 C_4  . 
\]
Furthermore, any conforming discretization of the problem will inherit 
the bounds from the continuous case. Within this framework, the
challenge is then to identify the proper norms 
for which \eqref{abstractiso} can be established 
and subsequently establishing efficient preconditioners for the required Riesz maps.
Multilevel algorithms that efficiently realize the mappings have been developed 
for standard spaces such as $H^1$, $H(\operatorname{div})$, and $L^2$ 
and weighted variants, c.f. e.g.~\cite{mardal2011preconditioning}. Furthermore, 
fractional multilevel solvers have been constructed in for example~\cite{bramble2000computational, baerland2018multigrid}.

We conclude the section with two numerical experiments which
demonstrate issues with establishing the preconditioners for Darcy-Stokes
problem based on the existing analysis. In~\cite{layton2002coupling,galvis2007non}
the well-posedness of the problem was established and suitable finite element 
methods developed. In particular, \cite{galvis2007non} derive error estimates
in parameter dependent norms that are robust with respect to the material
parameters. Example \ref{ex:darcy_wrong} shows that these norms are not
sufficient to establish robust preconditioners. 

\begin{example}[Darcy-Stokes preconditioner based on \cite{galvis2007non}]\label{ex:darcy_wrong}
  For simplicity and only to illustrate that the norms of \cite{galvis2007non} are not sufficient 
  for our preconditioning purposes, we consider \eqref{eq:darcy_stokes_strong} with $\BJS=1$, $\mu=1$ and 
  $\semi{\partial\Omega_{i, N}}=0$, $i=p, f$. 
  The setup of this and the subsequent experiments is described in detail below in 
  Remark \ref{rmrk:setup}.
  Error estimates for the finite element discretization of the system, which were
  robust in material parameters, were derived in \cite{galvis2007non} and
  the utilized weighted norms yield the following tentative guess for $X_{\eps}$ in \eqref{abstractiso}:
\[
\left (\mathbf{H}^1_{0, D}(\Omega_f) \cap \mathbf{L}_{t}^2(\Omega_f) \right ) \times \frac 1 {\sqrt K} \la{H}_{0, D}(\operatorname{div}, \Omega_p) \times L^2(\Omega_f) \times \sqrt{K} L^2(\Omega_p) \times H^{1/2}(\Gamma)
\]
  The resulting preconditioner is then:  
\begin{equation}\label{eq:darcy_wrong}
\mathcal{B} = \begin{pmatrix}
-{\Delta} + D\tG\dual  \tG &&&& \\
	   & K^{-1}\left({I} - \nablab \nablab \cdot\right) &&& \\
	   &&     I && \\
	   &&&     KI & \\ 
	   &&&&  (-\Delta +I)^{1/2}
	   \end{pmatrix}^{-1}.
\end{equation}
Here, the first four components are standard components for preconditioning 
of Darcy and Stokes problems, c.f. \cite{2018arXiv181200653B, vassilevski2013block, mardal2011preconditioning}.
The $\tG\dual \tG $ term is a benign additional term
for the Stokes problem that, in our experience, does not affect the performance of the preconditioner 
as the term only increases the diagonal dominance in parts of the matrix. 
The final block then reflects $H^{1/2}$ as the appropriate space for Lagrange
multiplier, cf. \cite{galvis2007non}. As the authors have recently developed efficient multilevel
algorithms for such fractional problems \cite{baerland2018multigrid},
all the building blocks of \eqref{eq:darcy_wrong} can be realized by order optimal
preconditioners that are spectrally equivalent with the corresponding Riesz mappings. 
However, we shall here use LU for simplicity and to put focus on the Riesz maps themselves (rather than
their numerical approximations).

Using discretization by stable \DSelm\ element, see \cite{galvis2007non},
Table \ref{tab:layton} shows the number of MinRes iterations preconditioned by
\eqref{eq:darcy_wrong} for $10^{-8}\leq K \leq 1$. Sensitivity to $K$ is evident.
We remark that a preconditioner based on a weighted multiplier space $\sqrt{K} H^{1/2}$ yields poorer performance.  
  \begin{table}[h]
    \begin{minipage}{0.46\textwidth}
  \begin{center}
    \footnotesize{
 \begin{tabular}{c|ccccc}
              \hline
\multirow{2}{*}{$K$} & \multicolumn{5}{c}{$h$}\\
              \cline{2-6}
              & $2^{-3}$ & $2^{-4}$ & $2^{-5}$ & $2^{-6}$ & $2^{-7}$ \\
              \hline
              1        & 67 & 73 & 74 & 76 & 77\\
              $10^{-2}$ & 109 & 149 & 159 & 163 & 158\\
              $10^{-4}$ & 122 & 260 & 396 & 476 & 523\\
              $10^{-6}$ & 122 & 246 & 419 & 675 & 997\\
              $10^{-8}$ & 122 & 229 & 379 & 554 & 735\\
\hline
 \end{tabular}
 }
  \end{center}
  \caption{MinRes iterations for Darcy-Stokes problem \eqref{eq:darcy_stokes_strong}
    using preconditioner \eqref{eq:darcy_wrong}.}
  \label{tab:layton}
    \end{minipage}
\begin{minipage}{0.53\textwidth}
  \begin{figure}[H]
\centering
\includegraphics[width=\textwidth]{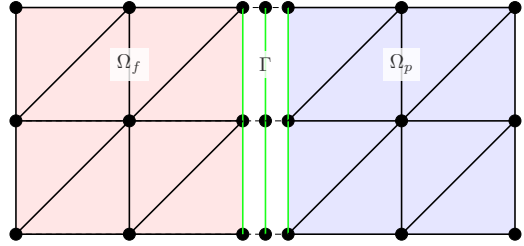}
\vspace{-20pt}
    \caption{Interface conforming tessellation $\mathcal{T}_h$ of
      domain $\Omega_f\cup\Omega_p$. Mesh of $\Gamma$ consists facets of elements
      in $\mathcal{T}_h$. Dashed line indicate correspondence of vertices.}
    \label{fig:mesh}
    \end{figure}
\end{minipage}
\end{table}

\end{example}

\begin{remark}[Common setup of experiments]\label{rmrk:setup}
  Throughout the paper we let $\Omega_f=\left[0, \tfrac{1}{2}\right]\times\left[0, 1\right]$,
  $\Omega_p=\left[\tfrac{1}{2}, 1\right]\times\left[0, 1\right]$ and
  $\Gamma=\left\{(x, y)\,|\,x=\tfrac{1}{2}, 0<y<1\right\}$ in the coupled problems.
  In the numerical experiments we consider a uniform triangulation $\mathcal{T}_h$ of $\Omega_p\cup\Omega_f$ into isosceles
  triangles with legs of size $h$. Further, the triangulation conforms to the interface in the
  sense that the no cell $K\in\mathcal{T}_h$ has its interior intersected by $\Gamma$. The mesh of $\Gamma$
  then consists of facets of $\mathcal{T}_h$, see also Figure \ref{fig:mesh}.
  The linear systems are assembled using the multiscale library FEniCS\textsubscript{ii} \cite{fenics_ii}, a module
  built on top of cbc.block \cite{Mardal2012} and the FEniCS framework \cite{fenics}.

  To solve the linear system $\mathcal{A}x=b$, a preconditioned minimal residual (MinRes) method
  is used with a random initial vector and convergence criterion based on
  relative preconditioned residual norm and tolerance $10^{-12}$. The blocks
  in the block diagonal preconditioners $\mathcal{B}$ are inverted exactly by
  LU factorization. The MinRes implementation as well as LU are provided by
  PETSc \cite{petsc}.

  For $s\in(-1, 1)$ the operators $-(\Delta + I)^s$ and $-(\Delta + I)_{00}^s$ on $\Gamma$
  are defined using an eigenvalue problem $-\Delta u + u=\lambda u$ with
  homogeneous Neumann respectively Dirichlet boundary conditions on $\partial\Gamma$.
  The discrete operator is computed by spectral decomposition as detailed 
  in \cite{kuchta2016preconditioners}. The discrete Laplacian on the piecewise constant field $Q_h$ is
  defined as
  \begin{equation}
  \label{disk:Laplace}
  (-\Delta p_h, q_h) = \sum_{E_I} \int_{E_I} \avg{h}^{-1}\jump{p} \jump{q}\mathrm{d}s + \sum_{E_D} \int_{E_D} h^{-1} p q\mathrm{d}s \quad p_h, q_h\in Q_h,
  \end{equation}
  where $E_I$ is the set of internal facets of the mesh, while $E_D$ is the 
  set of facets associated with the Dirichlet boundary. The average and jump
  operators are defined as $\avg{p}=\tfrac{1}{2}(p|_{K^{+}}+p|_{K^{-}})$, $\jump{p}=p|_{K^{+}}-p|_{K^{-}}$
  with $K^{+}$ and $K^{-}$ the two cells sharing the facet in $E_I$. Note that
  the set $E_D$ is empty for the operator $(-\Delta+I)^{s}$,  
  while $E_D$ is not empty  for $(-\Delta+I)_{00}^{s}$.

  Condition number estimate of the preconditioned linear system is obtained
  by solving the eigenvalue problem $\mathcal{A}x=\lambda \mathcal{B}^{-1}x$. If the
  number of unknowns is less than 8 thousand the entire spectrum is computed. Otherwise
  an iterative Krylov-Schur solver from SLEPc \cite{Hernandez:2005:SSF} is used to find
  the extreme eigenvalues. Here, the tolerance is set to $10^{-3}$.

  The finite element approximation error is computed by first interpolating the error into the space of
  discontinuous piecewise polynomials of degree $p+2$ where $p$ is the degree of the numerical
  solution. For the $H^{s}$ norm the error in interpolated in the space of piecewise linear
  polynomials while piecewise constant elements are used to discretize $H^s$.
\end{remark}

In~\cite{galvis2007non}, the results of \cite{layton2002coupling} were extended while paying
special attention to the material parameters and boundary conditions for the Lagrange multiplier. 
In Example \ref{ex:darcy_wrong}  
we showed that the results of~\cite{galvis2007non} cannot be directly extended to 
proper preconditioning within the operator preconditioning framework. One reason for this is
associated with the boundary conditions for the Lagrange multiplier. We therefore proceed
with a simplified example where this issue is addressed. 

\begin{example}[Boundary conditions in $H^s$]\label{ex:poisson_bcs}
Let $\Omega_1$, $\Omega_2$ be two domains with a common interface $\Gamma$. Further,
let $\partial\Omega_i\setminus\Gamma$ be decomposed into a Dirichlet boundary $\partial\Omega_{i,D}$ 
and a Neumann boundary $\partial\Omega_{i, N}$ such that $\semi{\partial\Omega_{i, D}} > 0$ for $i=1, 2$. We 
then consider the coupled problem\\
%
\begin{minipage}{0.56\textwidth}
  \flushleft{
\begin{equation}\label{eq:poisson_strong}
    \begin{aligned} 
      -\nabla\cdot(\kappa_i\nabla u_i) &= f_i&\text{ in } \Omega_i,\\
      u_i &= g_i &\text{ on }\partial\Omega_{i, D},\\
      \kappa_i\nabla u_i\cdot n_i &= h_i&\text{ on }\partial\Omega_{i, N},\\
      \kappa_1\nabla u_1\cdot n - \kappa_2\nabla u_2\cdot n &= h&\text{ on }\Gamma,\\
      u_1 - u_2 &= g&\text{ on }\Gamma,
    \end{aligned}
\end{equation}
}
\null\par\xdef\tpd{\the\prevdepth}
\end{minipage}
\begin{minipage}{0.42\textwidth}
\begin{figure}[H]
    \begin{center}
    \includegraphics[width=\textwidth]{img/inline_tex_poisson.tex}
     \end{center}
    \vspace{-15pt}
    \caption{Neumann-Dirichlet problem \eqref{eq:poisson_strong}. Interface intersects domain 
    with different boundary conditions on the subdomain boundaries.
    }
    \label{fig:domain}
\end{figure}
  \end{minipage}\\
%

\vspace{0.2cm}
  where in general $g\neq 0$ and thus for $\Omega=\Omega_1\cup\Omega_2$ if $u$ is such that $u|_{\Omega_i}=u_i$ then
  $u\notin H^1{(\Omega)}$ as $u$ is broken at $\Gamma$. In terms of finite element approximation we construct $u$ using function 
  spaces defined separately on $\Omega_1$ and $\Omega_2$. Based on intersection of $\Gamma$ with the boundary condition domains in \eqref{eq:poisson_strong} we shall investigate
  three different coupled problems with $\semi{\partial\Omega_{i, D}} > 0$, $i=1, 2$.
  In (DD) case $\semi{\partial\Omega_{i, N}}=0$, (ND) $\semi{\partial\Omega_{1, N}} > 0$, $\Gamma\cap\partial\Omega_{1, D}=\emptyset$
  and $\semi{\partial\Omega_{2, N}}=0$, (NN)
  $\semi{\partial\Omega_{i, N}} > 0$ and $\Gamma\cap\partial\Omega_{i, D}=\emptyset$. 
A schematic of the geometry of Neumann-Dirichlet problem (ND) considered further is shown in Figure \ref{fig:domain}.
A closely related application in cardiac modeling can be found in~\cite{tveito2017cell}.

Introducing a Lagrange multiplier $\lambda=-\kappa_1\nabla u_1\cdot n$, the weak form of \eqref{eq:poisson_strong} 
is given by operator 
\begin{equation}\label{eq:poisson_op}
    \mathcal{A}=\begin{pmatrix}
        -\kappa_1\Delta_1   & & T'_1\\
        & -\kappa_2\Delta_2 & -T'_2\\
        T_1 & -T_2 &
    \end{pmatrix},
\end{equation}
where $T_i=u_i|_\Gamma$ are the trace operators. In the following we shall construct 
preconditioners for $\mathcal{A}$ which are robust in discretization as well as the 
jump in $\kappa_i$ across the interface. 

Let us illustrate the construction by considering the (ND) problem first. 
Then, the left part of the problem is
\begin{equation}\label{eq:single}
\begin{aligned}
    -\kappa_1\Delta u_1 &= f_1&\mbox{ in }\Omega_1,\\
            u_1 &= g_1&\mbox{ on }\Gamma\cup\partial\Omega_{1, D},\\
            \kappa_1\nabla u_1\cdot n &= h_1 &\mbox{ on }\partial\Omega_{1, N},
\end{aligned}
\end{equation}
where the Dirichlet boundary condition on $\Gamma$ shall be enforced by 
a Lagrange multiplier. Then the trace operator maps 
$\sqrt{\kappa_1} H^1(\Omega_1) \rightarrow \sqrt{\kappa_1} H^{1/2}(\Gamma)$ 
because $\Gamma$ intersects only the Neumann part of $\partial \Omega_1$ and
the following preconditioner yields robust convergence 
\begin{equation}\label{eq:single_N}
    \mathcal{B}_N = \begin{pmatrix}
        -\kappa_1\Delta & \\
         & \kappa_1^{-1}(-\Delta+I)^{-1/2}\\
    \end{pmatrix}^{-1}.
\end{equation}
On the other hand, if we consider only the right part of the problem then the 
trace maps onto $\sqrt{\kappa_2} H_{00}^{1/2}(\Gamma)$ because $\Gamma$ intersects
Dirichlet boundary $\partial\Omega_{2, D}$ at both ends. The preconditioner
therefore becomes 
\begin{equation}\label{eq:single_D}
    \mathcal{B}_D = \begin{pmatrix}
        -\kappa_2\Delta & \\
         & \kappa_2^{-1}(-\Delta+I)_{00}^{-1/2}
    \end{pmatrix}^{-1}.
\end{equation}


Considering also the (DD) and (NN) problems we conclude that 
the operator $\mathcal{A}$ in \eqref{eq:poisson_op}
is an isomorphism $W\rightarrow W\dual$ with
\[
W=\sqrt{\kappa_1} H^1_{0, D}(\Omega_1)\times \sqrt{\kappa_2} H^1_{0, D}(\Omega_2)\times Q(\Gamma)
\]
where for the three cases we define $Q(\Gamma)$ as 
\begin{equation}\label{eq:cases}
    \begin{aligned}
      &(\text{DD}) \quad &\sqrt{\kappa_1^{-1}} H^{1/2}_{00}(\Gamma)\cap \sqrt{\kappa^{-1}_2} H^{1/2}_{00}(\Gamma),\\
      &(\text{NN}) \quad &\sqrt{\kappa_1^{-1}} H^{1/2}(\Gamma)\cap \sqrt{\kappa^{-1}_2} H^{1/2}(\Gamma),\\
      &(\text{ND}) \quad &\sqrt{\kappa_1^{-1}} H^{1/2}(\Gamma)\cap \sqrt{\kappa^{-1}_2} H^{1/2}_{00}(\Gamma).
    \end{aligned}
\end{equation}
The three preconditioners are then the Riesz maps with respect to the inner products of the corresponding
spaces. We remark that the ND case is the most challenging because of the mixed boundary condition
and is as such the focus of the following discussion. 

Using discretization in terms of \pelm{2}-\pelm{2}-\pelm{0} elements we demonstrate 
robustness of the canonical Riesz map preconditioners based on \eqref{eq:cases} 
by considering the preconditioned eigenvalue problems $\mathcal{A}x=\lambda\mathcal{B}^{-1}x$
where $\mathcal{B}=\text{diag}(-\Delta, S)$. For the (ND) problem operator $S$ is
\begin{equation}\label{eq:poisson_ND}
  S = \kappa_1^{-1}(-\Delta+I)^{-1/2} + \kappa_2^{-1}(-\Delta+I)_{00}^{-1/2}
  \end{equation}
and the Laplacian is defined as \eqref{disk:Laplace}.

In Table \ref{poisson_bcs_tab} we show the condition number of  \eqref{eq:poisson_op} in (ND) case
with different preconditioners. Only the (ND) preconditioner using \eqref{eq:poisson_ND}
can be seen to be robust both in the parameters and the discretization. The preconditioners
(DD) and (NN) seem $h$-robust when the parameters $\kappa_i$ are such that the effect
of the improper boundary conditions is relatively small.

Without including the results we remark that we have also verified that for the (DD) and (NN) problems  
the (DD) and (NN) preconditioners, respectively, are robust.  

\begin{table}
\label{poisson_bcs_tab}
\begin{center}
  \footnotesize{
  \begin{tabular}{c|ccccccc}
  \hline
  \multirow{2}{*}{$\kappa_2/\kappa_1$} & \multicolumn{6}{c}{$h$}\\
  \cline{2-8}
  & $2^{-1}$ & $2^{-2}$ & $2^{-3}$ & $2^{-4}$ & $2^{-5}$ & $2^{-6}$ & $2^{-7}$ \\
  \hline 
$10^6$  &     4.33&    5.07& 	5.36&	5.44&	5.46&	5.46&	5.46\\
$1$     &    4.40&    5.05&	5.33&	5.42&	5.45&	5.46&	5.46\\
 $10^{-6}$    &  5.49&  5.60&	5.68&	5.73&	5.75&	5.75&	5.75\\

  \hline 
$10^6$   & 5.00&	6.45&	7.47&	8.34&	9.18&	10.03&	10.89\\
$1$      & 4.69&	5.63&	6.32&	6.90&	7.46&	8.02 &	8.61\\
$10^{-6}$ & 5.49&	5.60&	5.68&	5.73&	5.75&	5.75 &	5.75\\

  \hline 
$10^6$  &  4.33&	5.07 &	5.36 &	5.44 &	5.46 &	5.46 &	5.46\\
$1$     &  5.36&	5.81 &	6.00 &	6.08 &	6.11 &	6.12 &	6.13\\
$10^{-6}$& 10.02&	13.01&	15.89&	18.80&	21.86&	25.13&	28.64\\

\hline
  \end{tabular}
  }
\end{center}
\caption {Spectral condition numbers for (ND) problem of \eqref{eq:poisson_strong}. Upper row (ND)
  preconditioner, middle row (DD) preconditioner and bottom row (NN)
  preconditioner.}
\end{table}
\end{example}

Parameter robust preconditioners for the Darcy-Stokes and Stokes-Navier
systems shall be derived within a general framework for coupled multiphysics/
multiscale problems.

\section{Abstract Framework}\label{sec:abstract}
Let us assume in the following that there are two saddle point problems which are both well-posed and 
have some of the Dirichlet boundary conditions enforced in terms of Lagrange multipliers on part of
the boundary. The unknowns can be either vector or scalar fields. Hence, there shall be two problems ($i=1,2$) of the form:
  Find $(u_i, p_i, \lambda_i) \in V_i \times Q_i \times \Lambda_i$ such that
\begin{equation} \label{eq:abstractsubprob}
    \AA_i  
    \left(
    \begin{array}{c} u_i \\ p_i \\ \lambda_i \end{array} \right) 
    =  
    \left(
    \begin{array}{ccc}
      A_i & B_i\dual & T_i\dual\\
      B_i & & \\
      T_i & &
    \end{array}
    \right) \left(
    \begin{array}{c}
      u_i \\ p_i \\ \lambda_i
    \end{array}
    \right) =  \left(\begin{array}{c}
      f_i \\ g_i \\ h_i
    \end{array} \right)
    \in \left(\begin{array}{c}
      V\dual_i \\ Q\dual_i \\ \Lambda\dual_i
    \end{array} \right).
\end{equation}
The well-posedness is guaranteed by the Brezzi conditions \cite{brezzi1974existence}, which in our setting read
\begin{subequations}
  \begin{align}
	\label{absbrezzi1}
	&(A_i u_i, u_i) \ge \alpha_i \|u_i\|^2_{Z_i}, \quad\forall u_i\in Z_i,  \\
	\label{absbrezzi2}
	&(A_i u_i, v_i) \le C_i \|u_i\|^2_{V_i} \|v_i\|^2_{V_i},\quad\forall u_i, v_i\in V_i, V_i\\  
	\label{absbrezzi3}
	&\sup_{u_i\in V_i} \frac{(B_i u_i, q_i) + (T_i u_i, \lambda_i)}{\|u_i\|_{V_i}}  \ge \beta_i(\|q_i\|^2_{Q_i} + \|\lambda_i\|^2_{\Lambda_i})^{1/2},\quad\forall q_i,\lambda_i \in Q_i, \Lambda_i, \\
	\label{absbrezzi4}
	&(B_i u_i, q_i) + (T_i u_i, \lambda_i)  \le D_i \|u_i\|_{V_i} (\|q_i\|^2_{Q_i} + \|\lambda_i\|^2_{\Lambda_i})^{1/2},\quad \forall u_i, q_i, \lambda_i \in V_i, Q_i, \Lambda_i,
  \end{align}        
\end{subequations} 
where $Z_i = \{u_i \in V_i \ | \ (B_i u_i, q_i) + (T_i u_i, \lambda_i) = 0,\quad \forall q_i, \lambda_i \in Q_i, \Lambda_i \} $.
We shall also consider the following condition, which is stronger than \eqref{absbrezzi1},  
but more commonly considered in single-physics problems with Dirichlet boundary conditions enforced in the standard way. Namely,  
\begin{eqnarray}
  \label{absbrezzi1kerB}
	(A_i u_i, u_i) &\ge& \alpha_i \|u_i\|^2_{Z_i},\quad\forall u_i\in \{u_i \in V_i \ | \ (B_i u_i, q_i) = 0,\quad \forall q_i \in Q_i \}.
\end{eqnarray} 
The Brezzi conditions ensure that both 
\[
\|A_i\|_{\mathcal{L}((V_i\times Q_i\times \Lambda_i),(V_i\times Q_i\times\Lambda_i)\dual)} \quad \mbox{and} \quad \|(A_i)^{-1}\|_{\mathcal{L}((V_i\times Q_i\times \Lambda_i)\dual,(V_i\times Q_i\times \Lambda_i))}  
\]
are bounded and the last bound can alternatively be written in the following form, which will be used later,   
\begin{equation}
\label{wellposednessdfn} \|u_i\|_{V_i} + \|p_i\|_{Q_i} + \|\lambda_i\|_{\Lambda_i}
\leq E_i \left ( \|f_i\|_{{V_i}\dual} + \|g_i\|_{Q_i\dual} + \|h_i\|_{\Lambda_i\dual} \right ).
\end{equation}
Here, $E_i$ depends only on $\alpha_i$, $\beta_i$, $C_i$ and $D_i$. 

Let us then consider the existence and uniqueness of the coupled problem:
Find $(u_1, p_1, u_2, p_2, \lambda) \in V_1 \times Q_1 \times V_2 \times Q_2 \times \Lambda_1\cap \Lambda_2$ such that
\begin{equation}
\label{eq:abstractcoupledprob}
\AA \left (
    \begin{array}{c}
      u_1 \\ u_2 \\ p_1 \\ p_2 \\ \lambda
    \end{array}
\right ) =  
\left (\begin{array}{ccccc}
           A_i & & B_i\dual & & T_1\dual\\
               & A_2 & & B_2\dual & T_2\dual \\
           B_1  & & & & \\
               & B_2 & & & \\
           T_1 & T_2 & & & \\
    \end{array}
\right )   \left (
    \begin{array}{c}
      u_1 \\ u_2 \\ p_1 \\ p_2 \\ \lambda
    \end{array}
\right ) =  \left (\begin{array}{c}
      f_1 \\ f_2 \\ g_1 \\ g_2 \\ h
    \end{array} \right )
\in
\left (\begin{array}{c}
      V_1\dual \\ V_2\dual \\ Q_1\dual \\ Q_2\dual \\ \Lambda\dual_1+ \Lambda\dual_2  
    \end{array} \right ).
\end{equation}
We remark that $T_i: V_i \rightarrow \Lambda_i\dual$ and 
hence $T_1 u_1 + T_2 u_2 \in \Lambda_1\dual + \Lambda_2\dual$. Therefore,  
$\lambda\in \Lambda_1 \cap \Lambda_2$ since
$(\Lambda_1\cap \Lambda_2)\dual=\Lambda\dual_1+ \Lambda\dual_2$.

Our main result concerning \eqref{eq:abstractcoupledprob} is stated in the
following theorem.

{
\begin{theorem}\label{thm:abstractcoupledprob}
  Suppose that the problems \eqref{eq:abstractsubprob} satisfy the Brezzi conditions\\
  \eqref{absbrezzi1}--\eqref{absbrezzi4} in $V_i \times Q_i \times \Lambda_i$, $i=1, 2$
  and the coercivity condition \eqref{absbrezzi1kerB}. Then the coupled problem \eqref{eq:abstractcoupledprob} is well posed in
$W=V_1 \times Q_1 \times V_2 \times Q_2 \times (\Lambda_1 \cap \Lambda_2)$ in the sense that 
\[
\|\AA\|_{\mathcal{L}(W,W\dual)} \quad \mbox{and} \quad \|\AA^{-1}\|_{\mathcal{L}(W\dual,W)}  
\]
are bounded by some positive constant
$C$ depending only on the Brezzi constants of problems \eqref{eq:abstractsubprob}. 

\end{theorem}
}

\begin{proof}[Proof]
  We verify the Brezzi conditions for \eqref{eq:abstractcoupledprob}
  in the form
  
  \[
  \AA = 
  \left(\begin{array}{cc}
  A & B\dual \\
  B &     \\
    \end{array}
  \right ),
  \mbox{ where }
	A = \left (\begin{array}{cc}
           A_1 & \\ 
           & A_2 
    \end{array}
\right) \quad \mbox{ and } \quad 
B = 
\left(\begin{array}{cc}
B_1  &  \\
& B_2  \\
T_1 & T_2  \\
\end{array}
\right), 
\]
that is, by considering $\AA$ as an operator on $V \times Q$ where $V = V_1 \times V_2$ and $Q = Q_1 \times Q_2 \times \left ( \Lambda_1 \cap \Lambda_2 \right )$. The boundedness of $A$ follows from \eqref{absbrezzi1} because for any $u_1, v_1 \in V_1, u_2, v_2 \in V_2$ we have that
\begin{equation*}
\begin{split}
  (A (u_1, u_2), (v_1, v_2)) &= (A_1 u_1, v_1) + (A_2u_2, v_2)\\
  & \leq \max(C_1, C_2) \|(u_1, u_2)\|_{V_1\times V_2} \|(v_1, v_2)\|_{V_1 \times V_2}.
\end{split}
\end{equation*}
Similarly, the boundedness of $B$ follows from \eqref{absbrezzi4}: 	
\begin{align*}
  (B_1u_1, p_1) &+ (T_1 u_1, \lambda) + (B_2u_2, p_2) + (T_2 u_2, \lambda) \leq\\
	 &\leq D_1 \|u_1\|_{V_1} \left( \|p_1\|^2_{Q_1} + \|\lambda\|^2_{\Lambda_1}\right)^{\frac 1 2} 
	 +D_2 \|u_2\|_{V_2} \left( \|p_2\|^2_{Q_2} + \|\lambda\|^2_{\Lambda_2}\right)^{\frac 1 2}  \\
   & \leq  \max(D_1, D_2) \left (\|u_1\|_{V_1}^2 + \|u_2\|_{V_2}^2 \right )^{\frac 1 2} \left ( \|p_1\|^2_{Q_1} + \|\lambda\|^2_{\Lambda_1} + \|p_2\|^2_{Q_2} + \|\lambda\|^2_{\Lambda_2} \right )^{\frac 1 2} \\
 & = \max(D_1, D_2) \|(u_1, u_2)\|_{V_1 \times V_2} \left ( \|(p_1, p_2)\|^2_{Q_1 \times Q_2} + \|\lambda\|^2_{\Lambda_1 \cap \Lambda_2}\right )^{\frac 1 2} 
\end{align*}
for all $(u_1, u_2, p_1, p_2, \lambda) \in V_1 \times V_2 \times Q_1 \times Q_2 \times (\Lambda_1 \cap \Lambda_2)$, where the second inequality is the Cauchy-Schwarz inequality. Hence $A$ and $B$ are both bounded, with boundedness constants depending only on those of the subproblems.

For coercivity, note that because $
   B(u_1, u_2) = \left (
     \begin{array}{c}
       B_1u_1 \\ B_2u_2 \\ T_1u_1 + T_2u_2
     \end{array} \right ),$ we have that  $$\ker B = \left (\ker B_1 \times \ker B_2\right ) \cap \ker \left ( \begin{array}{cc} T_1 & T_2 \end{array} \right) \subset \ker B_1 \times \ker B_2.$$
     By assumption \eqref{absbrezzi1kerB}, 
     $A_i$ is coercive on $\ker B_i$ with coercivity constant $\alpha_i$,
     meaning that $(A_iu_i, u_i) \geq \alpha_i \|u_i\|^2$ for all $u_i \in \ker B_i$. Hence for any $(u_1, u_2) \in \ker B$, 
     \begin{equation*}
       \begin{split}
         (A(u_1, u_2), (u_1, u_2))_{V_1 \times V_2} &= (A_1u_1, u_1)_{V_1} + (A_2u_2, u_2)_{V_2} \\
         & \geq \alpha_1 \|u_1\|_{V_1}^2  + \alpha_2 \|u_2\|_{V_2}^2 \\
         & \geq \min(\alpha_1, \alpha_2) \|(u_1, u_2)\|^2_{V_1 \times V_2}.
       \end{split}
\end{equation*}
Thus $A$ is coercive on $\ker B$ with constant $\min(\alpha_1, \alpha_2)$.
 
To prove the inf-sup condition, let $\Rszi{{Q_i}}: Q_i \to Q_i{\dual}$, 
$\Rszi{{\Lambda_i}}: \Lambda_i \to \Lambda_i{\dual}$ be the inverse Riesz maps of their corresponding spaces. By the Riesz representation theorem, this is an isometry between $Q$ and $Q\dual$, meaning that $(\Rszi{Q}q, q) = \|q\|^2_{Q}$.

Given $(q_1, q_2, w) \in Q_1 \times Q_2 \times \Lambda\dual_1 \cap \Lambda\dual_2$, let $u^*_i, p^*_i, \lambda^*_i$ be the solution of \[\left (
    \begin{array}{ccc}
      A_i & B_i\dual & T_i\dual\\
      B_i & & \\
      T_i & &
    \end{array}
  \right )
  \left (
    \begin{array}{c}
      u^*_i \\ p^*_i \\ \lambda^*_i
    \end{array}
\right ) = \left ( \begin{array}{c}
      0 \\ \Rszi{{Q_i}} q_i \\ \Rszi{{\Lambda_i}} w
    \end{array} 
\right ) \: \text{ for $i=1, 2$ .}
\] 

Considering $u^*_i = u^*_i(q_i, w)$ as a function of $q_i, w$, by \eqref{wellposednessdfn} we have that

\begin{equation} \label{eq:subprobconstantdfn}
\|u^*_i\|^2_{V_i} \leq 2E^2_i \left ( \|\Rszi{{Q_i}}q_i\|^2_{{Q_i}\dual} +\|\Rszi{{\Lambda_i}}w\|^2_{{\Lambda_i}\dual} \right ) = 2E^2_i \left ( \|q_i\|^2_{Q_i} +\|w\|^2_{\Lambda_i} \right )
\end{equation}
for any $(q_i, w)\in Q\dual_i \times \Lambda\dual_i$. 

Further, we have that $(B_iu^*_i, q_i) + (T_iu^*_i, w) = (\Rszi{{Q_i}}q_i, q_i) + (\Rszi{{\Lambda_i}}w, w) = \|q_i\|^2_{Q_i} + \|w\|^2_{\Lambda_i}$. Combining the results

\begin{eqnarray*}
  & \sup_{(u_1, u_2)\in V_1\times V_2}\frac{(B_1u_1, q_1) + (B_2u_2, q_2) + (T_1u_1, w) + (T_2u_2, w)}{(\|u_1\|^2_{V_1} + \|u_2\|^2_{V_2})^{\half}} \\
  \geq & \frac{(B_1u^*_1, q_1) + (B_2u^*_2, q_2)  + (T_1u^*_1, w) + (T_2u^*_2, w)}{(\|u^*_1\|^2_{V_1} + \|u^*_2\|^2_{V_2})^{\half}} \\
  \geq & \frac 1 E \frac{\|q_1\|^2_{Q_1} + \|q_2\|^2_{Q_2} + \|w\|^2_{\Lambda_1} + \|w\|^2_{\Lambda_2}}{(\|q_1\|^2_{Q_1} + \|q_2\|^2_{Q_2} + \|w\|^2_{\Lambda_1} + \|w\|^2_{\Lambda_2})^{\half}} \\
  = & \frac 1 E (\|q_1\|^2_{Q_1} + \|q_2\|^2_{Q_2} + \|w\|^2_{\Lambda_1} + \|w\|^2_{\Lambda_2})^{\half} \\
  = & \frac 1 E (\|q_1\|^2_{Q_1} + \|q_2\|^2_{Q_2} + \|w\|^2_{\Lambda_1 \cap \Lambda_2})^{\half},
\end{eqnarray*}
where  $E = 2\max(E_1, E_2)^2$. Hence we have the desired inf-sup condition.

As all the Brezzi conditions hold, we have the bound \eqref{wellposednessdcoupled}.
\end{proof}

We remark that the boundedness of the inverse map can be written as the stability estimate
\begin{equation}
  \label{wellposednessdcoupled}
  \begin{split}
\|u_1\|_{V_1} + \|p_1\|_{Q_1} +& \|u_2\|_{V_2} + \|p_2\|_{Q_2} + \|\lambda\|_{\Lambda_1 \cap \Lambda_2}\\
&\leq C
\left ( \|f_1\|_{V_1\dual} + \|g_1\|_{Q_1\dual} + 
\|f_2\|_{V_2\dual} + \|g_2\|_{Q_2\dual} + \|h\|_{\Lambda_1\dual + \Lambda_2\dual} \right )
\end{split}
\end{equation}
which corresponds the results of~\cite{galvis2007non} for the Darcy--Stokes problem 
except that instead of 
$\|h\|_{\Lambda_1\dual + \Lambda_2\dual}$
they use
$\max(\|h\|_{\Lambda_1\dual}, \|h\|_{\Lambda_2\dual})$.

\section{Robust Preconditioners for the Darcy--Stokes system}\label{sec:darcy-stokes}
In this section we derive parameter robust preconditioners for the Darcy-Stokes
problem \eqref{eq:darcy_stokes_strong} within the framework presented in \S \ref{sec:abstract}.
In particular, the Lagrange multiplier enforcing the mass-conservation condition
\eqref{eq:darcystokesmass} shall be established in a suitable intersection space
of fractional spaces. As we saw in Example \ref{ex:darcy_wrong}, preconditioners that  ignore the structure of
the multiplier space are not robust with respect to certain parameter variations.
Furthermore, in Example \ref{ex:poisson_bcs} we saw that setting appropriate boundary conditions 
for the Lagrange multipliers is a delicate subject and that the condition affects the performance if not done correctly.

Let Dirichlet conditions be applied on $\partial \Omega_{f, D}, \partial \Omega_{p, D}$,
and Neumann conditions on $\partial \Omega_{f, N}, \partial \Omega_{p, N}$, c.f. Figure~\ref{fig:DSdomains}.
Suppose also that
$\Gamma\cap \partial\Omega_{i, D}=\emptyset$, $i=p, f$ and
$\semi{\partial\Omega_{i, D}} > 0$.  We define
\begin{align*}
  \mathbf{V}_f =& \smuH,  \\
  Q_f =& \ismuL, \\
  \mathbf{V}_p =& \isKHd, \\
  Q_p =& \sKL, \\
  \Lambda =& \DSmultf \cap \DSmultpD.
\end{align*}

Following \cite{galvis2007non,layton2002coupling} the weak formulation of the
Darcy-Stokes problem \eqref{eq:darcy_stokes_strong} reads:
Find $(\uf, \up, p_f, p_p, \lambda)$ in $\mathbf{W}=\mathbf{V}_f\times \mathbf{V}_p\times Q_f\times Q_p \times \Lambda$ 
such that for all $(\vf, \vp) \in \mathbf{V}_f\times \mathbf{V}_p$ and
all $(q_f, q_p, w)\in Q_f\times Q_p\times \Lambda$
\begin{equation}
  \label{eq:darcy_stokes_weak}
\begin{aligned}
a((\uf, \up), (\vf, \vp)) + b((\vf, \vp), (p_f, p_p, \lambda)) &= f((\vf, \vp)),\\ 
  b((\uf, \up), (q_f, q_p, w)) &= g((q_f, q_p, w)),
\end{aligned}
\end{equation}
where
\begin{align*} 
  a((\uf, \up), (\vf, \vp)) &=  \mu (\nabla \uf, \nabla \vf)_{\Omega_f} + D (\tG\uf, \tG\vf)_\Gamma  + K^{-1} (\up, \vp)_{\Omega_p}, \\  
  b((\uf, \up), (q_f, q_p, w)) &=  (\nabla \cdot \uf, q_f)_{\Omega_f} + (\nabla \cdot \up, q_p)_{\Omega_p} +  (T_{n}\uf, w)_{\Gamma} -(T_{n}\up, w)_{\Gamma}, \\
  f((\vf, \vp)) &= (\ff, \vf)_{\Omega_f} + (\hf, \vf)_{\partial\Omega_{f, N}}+(\hp, \vp \cdot \boldsymbol{n})_{\partial\Omega_{p, N}}, \\
  g((q_f, q_p, w)) &= (\fp, q_p)_{\Omega_p}.
\end{align*}
We remark that the Lagrange multiplier $\lambda$ is defined as a normal
component of the traction force, i.e. $\lambda=p_p$ and by \eqref{eq:darcystokesstress} also
$\lambda=-\mu\tfrac{\partial \uf}{\partial\nG}\cdot\nG+p_f$.
The coefficient matrix of the left-hand side of \eqref{eq:darcy_stokes_weak} is 
\begin{align}
\label{coeff:darcy:stokes}
\AA =    \left( \begin{array}{cc|ccc}
     -\mu {\Delta} + D \tG\dual \tG  &   &-\nabla & & T_{ n}\dual \\ 
      & K^{-1} I  & & -\nabla & - T_{n}' \\ \hline
     \nabla \cdot&  &  & &  \\ 
     & \nabla \cdot  &  & &  \\ 
     T_{n} & -T_n &  & &   
     \end{array} \right).
\end{align}

In Example \ref{ex:darcy_wrong}, we saw that the efficiency of the preconditioning of the system
\eqref{eq:darcy_stokes_strong} varied substantially with the material parameters even
though the Stokes block and the Darcy block were preconditioned with  appropriate preconditioners.
We next demonstrate that robustness with respect to mesh resolution and variations in
material parameters can be obtained by choosing properly weighted fractional spaces for posing
the Lagrange multiplier. The preconditioner shall be of the form
{
\begin{align}
\label{eq:B_darcy_stokes}
\BB =    \left( \begin{array}{ccccc}
     -\mu {\Delta} + D \tG\dual \tG &   & & &  \\ 
      & K^{-1} \left(I - \nablab \nablab \cdot\right) &&& \\
      && \mu^{-1}I && \\
      &&& K I &  \\
      &&&& S
     \end{array}
     \right)^{-1},
\end{align}
}
where
\[
S = K(-\Delta+I)_{00}^{1/2} + \mu^{-1}(-\Delta+I)^{-1/2}.
\]

We remark that \eqref{eq:B_darcy_stokes}
is similar to the preconditioner proposed in Example \ref{ex:darcy_wrong} except for
the multiplier block where two fractional operators with different boundary conditions
and weighting by $\mu$ and $K$ form the Schur complement preconditioner at the interface. 

Before we analyze the preconditioner \eqref{eq:B_darcy_stokes} we consider its performance in 
Example \ref{ex:darcy_right} along with suitable assumptions in  
\ref{assumption:stokesreg} and \ref{assumption:darcyreg}. We start with boundary conditions
commonly met in practical applications rather than cases of homogeneous Dirichlet 
conditions that are often utilized for theoretical purposes. 

\begin{example}[Robust Darcy-Stokes preconditioner]\label{ex:darcy_right}
We consider \eqref{eq:darcy_stokes_strong} in case $\Gamma\cap \partial\Omega_{i, D}=\emptyset$, $i=p, f$ and
$\semi{\partial\Omega_{i, D}} > 0$, cf. Figure \ref{fig:DSdomains}. Using
\eqref{eq:B_darcy_stokes} and discretization in terms of \DSelm\ elements Figure 
\ref{fig:B_darcy_stokes} shows that the preconditioner is robust in discretization
parameter $h$ as well as variations in $\mu$, $K$ and $\BJS$.
\begin{figure}
  \centering
  \includegraphics[height=0.5\textwidth]{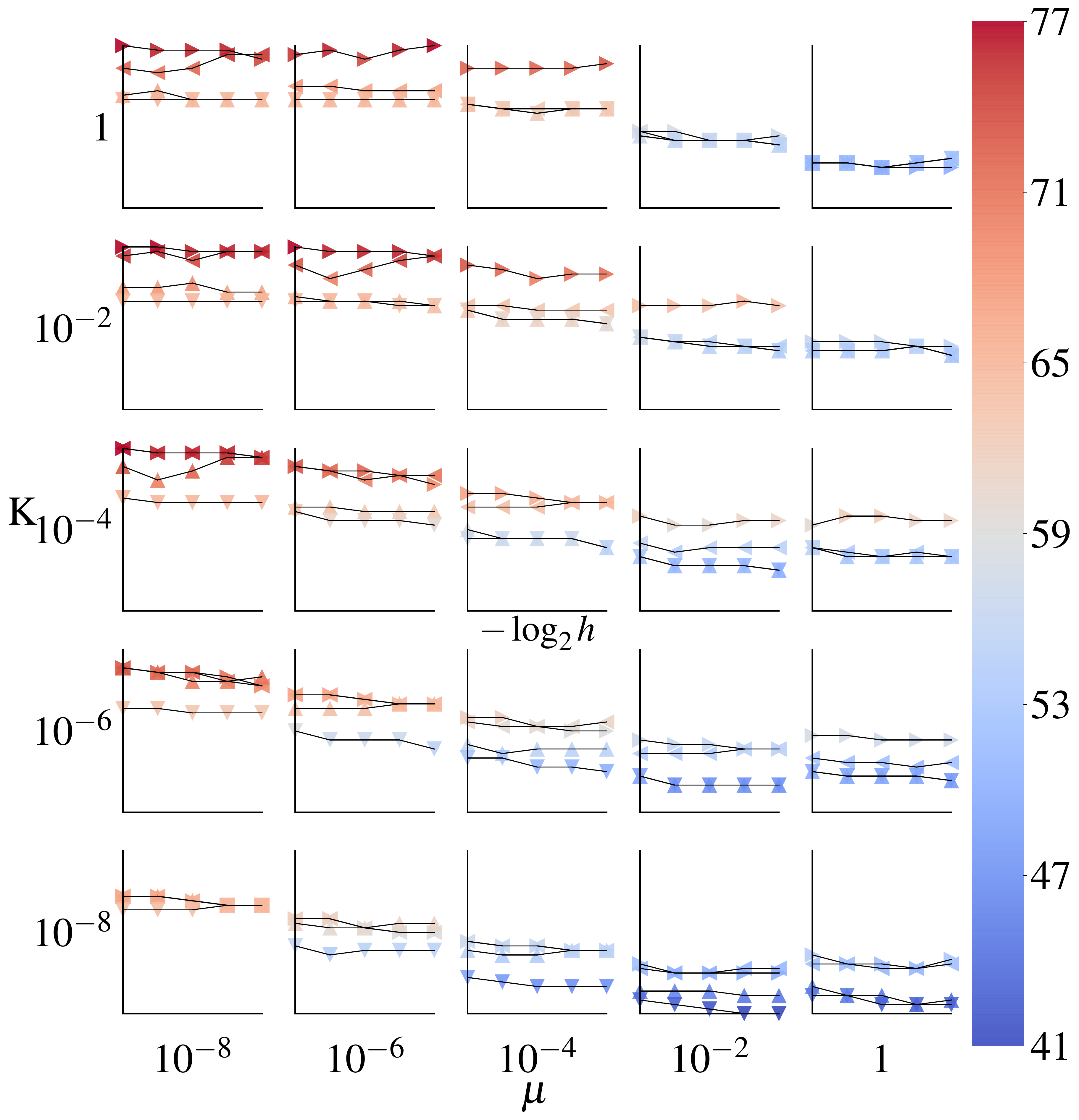}
  \includegraphics[height=0.5\textwidth]{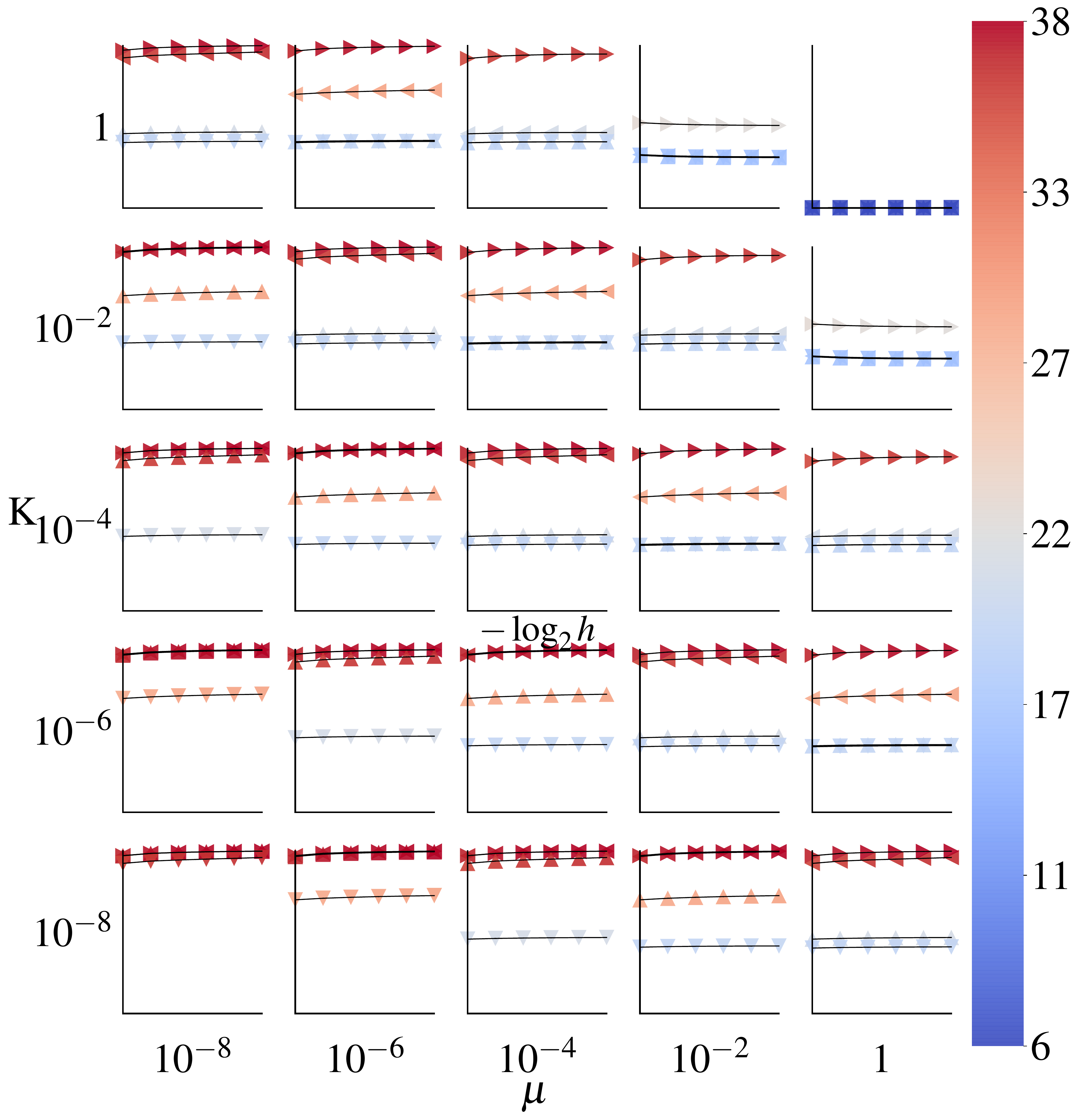}
    \vspace{-15pt}  
  \caption{Robust Darcy-Stokes preconditioner \eqref{eq:B_darcy_stokes}
in case $\Gamma\cap \partial\Omega_{i, D}=\emptyset$, $i=p, f$ and
$\semi{\partial\Omega_{i, D}} > 0$. (Left) Number of preconditioned MinRes iterations.
(Right) Spectral condition number of the preconditioner problem. The coarsest mesh for left plot has
$h=2^{-3}$ while $h=2^{-1}$ in the right plot. For fixed $K$, $\mu$ subplots the horizontal axis is scaled as $-log_2 h$ so that the system size grows from left to right. Values of $\alpha_{\text{BJS}}=10^{-6}, 10^{-4}, 10^{-2}, 1$ 
    are encoded with markers $\triangledown$, $\triangle$, $\triangleleft$, $\triangleright$.
  }
  \label{fig:B_darcy_stokes}
\end{figure}

The setup of the experiment and the solvers was summarized in the previous Remark
\ref{rmrk:setup}. Further, the case of homogeneous Dirichlet boundary conditions
is addressed in Remark \ref{rmrk:DS_DD}.
\end{example}

In order to apply Theorem \ref{thm:abstractcoupledprob} to prove that the preconditioner
\eqref{eq:B_darcy_stokes} is parameter robust, we require that the Stokes and Darcy subproblems
are individually well-posed in a specific form. The following assumptions specify the requirement.

\begin{assumption}[Stokes subproblem]\label{assumption:stokesreg}
  Let $\Omega_f\subset\reals^2$ be a bounded domain with boundary decomposition
  $\partial\Omega_f=\Gamma\cup\partial\Omega_{f, D}\cup\partial\Omega_{f, N}$ where the
  different parts are assumed to be of non-zero measure and $\partial\Omega_{f, D}\cap\Gamma=\emptyset$,
  cf. Figure \ref{fig:DSdomains}. 
  We consider the Stokes problem \eqref{eq:darcystokes1}-\eqref{eq:darcystokes2} with
  $\uf^0=\la{0}$, $\hf=\la{0}$ and the boundary conditions
  \begin{subequations}
    \begin{align}
      \label{eq:bab_stokes_n}
      \uf \cdot \nG &= g_n &\mbox{ on }\Gamma,\\
      \label{eq:bab_stokes_t}
  \boldsymbol{\tau}\cdot(\mu\nabla\uf-p_f I)\cdot \nG + \uf \cdot \boldsymbol{\tau} &= 0 &\mbox{ on }\Gamma,
\end{align}
  \end{subequations}
  where \eqref{eq:bab_stokes_n} shall be enforced by Lagrange multiplier.
  Let $\mathbf{W}=\mathbf{V}\times Q \times \Lambda$ with 
  $\mathbf{V} = \smuH$, $Q=\mu^{-\half}L^2(\Omega_f)$, $\Lambda=\mu^{-\half}H^{-\half}(\Gamma)$ and $\ff \in \mathbf{V}\dual$, $g_n\in \Lambda\dual$.
  We define $a:\mathbf{W} \times \mathbf{W}\rightarrow\reals$ and $L:\mathbf{W}\rightarrow\reals$
  as 
\begin{equation}\label{eq:bab_stokes_op}
   \begin{aligned} 
     a((\uf, p_f, \lambda), (\vf, q_f, w)) =& \mu (\nabla \uf, \nabla \vf) + (\tG\uf, \tG\vf)_{\Gamma} + (p_f, \nabla \cdot \vf)  + \\
                                           &+ (\nabla \cdot \uf, q_f)
                                           +(\lambda, T_n\vf)_{\Gamma} + (w, T_n\uf)_{\Gamma},\\
      L((\vf, q_f, w)) =& (\ff, \vf) + (g_n, w).
   \end{aligned}
\end{equation}
We assume that the Brezzi conditions are met by \eqref{eq:bab_stokes_op} 
so that the problem: Find $(\uf, p_f, \lambda)\in \mathbf{W}$ such that
\[
a((\uf, p_f, \lambda), (\vf, q_f, w)) = L((\vf, q_f, w)),\quad\forall(\vf, q_f, w) \in \mathbf{W}
\]
is well-posed. In particular, the Brezzi conditions ensure the following stability estimate that will be used later
\begin{equation}\label{eq:stokes_estimate}
\| (\uf, p_f, \lambda)\|_{\mathbf{W}} \leq C\left ( \|\ff\|_{\mathbf{V}^{\prime}} + \|g_n\|_{\Lambda\prime}  \right ).
\end{equation}
Here the constant $C$ is independent of the parameters.
\end{assumption}

We remark that the condition \eqref{eq:bab_stokes_t} is a special case of the
Beavers-Joseph-Saffmann condition, with $D=1$. We consider this simplification
as the results of Example \ref{ex:darcy_right} show that sensitivity of the coupled
problem to variations of $D$ is small.

\begin{assumption}[Darcy subproblem]\label{assumption:darcyreg}
  Let $\Omega_p\subset\reals^2$ be a bounded domain with boundary decomposition
  $\partial\Omega_p=\Gamma\cup\partial\Omega_{p, D}\cup\partial\Omega_{p, N}$ where the
  components are assumed to be of non-zero measure and $\partial\Omega_{p, D}\cap\Gamma=\emptyset$. 
  We consider the Darcy problem \eqref{eq:darcystokes3}-\eqref{eq:darcystokes4} with
  $u_p^0=0$, $h_p=0$ and the boundary condition
  \[
      \up \cdot \nG = g_n \mbox{ on }\Gamma,\\
  \]
  which shall be enforced by Lagrange multiplier.

  Let $\mathbf{V} = \isKHd$, $Q=\sKL$, $\Lambda=\sqrt{K}H^{\half}_{00}(\Gamma)$  and $\mathbf{W}=\mathbf{V}\times Q\times \Lambda$. 
  Further let $\fp \in Q\dual$,  $g_n\in \Lambda\dual$ and let us define
  
  \begin{equation}\label{eq:bab_darcy_op}
    \begin{aligned}
      a((\up, p_p, \lambda), (\vp, q_p, w)) =& K^{-1} (\up, \vp) + (p, \nabla \cdot \vp)  + \\
      &+(\nabla \cdot \up, q)  + (T_n \up, w)_{\Gamma} + (\lambda, T_n\vp)_{\Gamma},\\
L((\vp, q_p, w)) =& (f, q_p) + (g_n, w)_{\Gamma}.
    \end{aligned}
\end{equation}  
  We assume that \eqref{eq:bab_darcy_op} satisfies the Brezzi conditions such that the problem:
  Find $(\up, p_p, \lambda)\in \mathbf{W}$ such that
  \[
  a((\up, p_p, \lambda), (\vp, q_p, w)) = L((\vp, q_p, w)),\quad\forall (\vp, q_p, w)\in \mathbf{W}
  \]
  is well-posed. In particular, the Brezzi conditions ensure the following stability 
  estimate that will be used later
  \begin{equation}\label{eq:darcy_estimate}
  \|(\up, p_p, \lambda)\|_{\mathbf{W}} \leq C\left ( \|\fp\|_{Q\dual} +  \|g_n\|_{\Lambda\dual} \right )
  \end{equation}
  with $C$ independent of $K$.
\end{assumption}

We shall not prove Assumptions \ref{assumption:stokesreg} and \ref{assumption:darcyreg}. 
However, Examples \ref{ex:stokesreg} and \ref{ex:darcyreg} will provide numerical evidence in their 
support. In particular, the experiments show that the condition numbers of the discretized
systems do not vary significantly with discretization or material parameters if
preconditioned with the norms of the assumption.

\begin{example}[Demonstration of Assumption \ref{assumption:stokesreg}]\label{ex:stokesreg}
  Let $\Omega_f=\left[0, 1\right]^2$  with $\Gamma=\set{(x, y)\in\partial{\Omega_f}\,|\,x=0}$ and 
    $\partial\Omega_{f, D}=\set{(x, y)\in\partial{\Omega_f}\,|\,x=1}$. We demonstrate that
      Assumption \ref{assumption:stokesreg} holds by considering the spectra of the
      preconditioned problem $\mathcal{A}x=\beta\mathcal{B}^{-1}x$ where $\mathcal{A}$
      is the operator due to the bilinear form in \eqref{eq:bab_stokes_op} and $\mathcal{B}$
      is the Riesz map preconditioner induced by the space $\mathbf{W}$, i.e.
      \begin{equation}\label{eq:stokes_example}
        \begin{aligned}
    \mathcal{A} &= \begin{pmatrix}
        -\mu{\Delta} + \tG\dual\tG & -\nabla & T_n\dual\\
        \nabla\cdot & &\\
        T_n & &\\
    \end{pmatrix},\\
\mathcal{B} &= \begin{pmatrix}
      -\mu{\Delta} + \tG\dual\tG &  &\\
    & \mu^{-1}I &\\
    & & \mu^{-1}{(-\Delta + I)}^{-1/2}
\end{pmatrix}^{-1}.
\end{aligned}
\end{equation}
In order to illustrate the importance of the boundary conditions we shall in addition consider
the preconditioner $\mathcal{B}_{00}$ which differs from \eqref{eq:stokes_example}
by using $\mu^{-1}{(-\Delta + I)}_{00}^{-1/2}$ for the multiplier block.
      
\begin{minipage}{0.45\textwidth}
      \centering
      \footnotesize{
      \begin{tabular}{c|ccc||c}
        \hline
        \multirow{2}{*}{$h$} & \multicolumn{3}{c||}{$\mathcal{B}\mathcal{A}$} & {$\mathcal{B}_{00}\mathcal{A}$}\\
        \cline{2-5}
        & $\mu=1$ & $10^{-4}$ & $10^{-8}$ & $1$\\
        \hline
$2^{-1}$ & 10.19 & 13.45 & 13.46 & 9.29  \\
$2^{-2}$ & 10.17 & 13.41 & 13.41 & 10.21 \\
$2^{-3}$ & 10.17 & 13.40 & 13.40 & 11.06 \\
$2^{-4}$ & 10.17 & 13.39 & 13.39 & 11.92 \\
$2^{-5}$ & 10.17 & 13.39 & 13.39 & 12.80 \\
$2^{-6}$ & 10.17 & 13.39 & 13.39 & 13.71 \\
$2^{-7}$ & 10.17 & 13.39 & 13.39 & 14.64 \\
\hline
      \end{tabular}
        }
      \captionof{table}{Spectral condition numbers of preconditioned Stokes
        problem \eqref{eq:bab_stokes_op}. $\mathcal{B}$ is robust in $h$ and $\mu$.
        Results with $\mathcal{B}_{00}$ show
        that $H_{00}^{-1/2}$ is not suitable if $\Gamma\cap\partial\Omega_{f, D}=\emptyset$.}
  \label{tab:stokesreg}  
\end{minipage}
\begin{minipage}{0.49\textwidth}
  \begin{figure}[H]
      \centering        
\includegraphics[width=\textwidth]{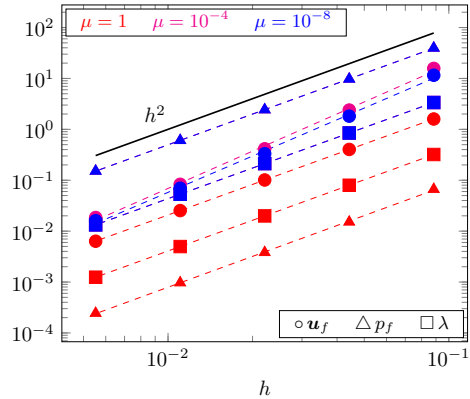}
      \vspace{-25pt}
      \caption{Approximation errors of Stokes problem \eqref{eq:bab_stokes_op}
        measured in norm due to $\mathcal{B}^{-1}$. Discretization by \Selm elements.}
      \label{fig:stokesreg}
  \end{figure}
\end{minipage}
\vspace{0.2cm}
  
  Table \ref{tab:stokesreg} lists the condition numbers of the preconditioned Stokes system
  discretized by \Selm\ elements for different values of $\mu$. The results
  are bounded,  indicating that the Brezzi conditions \eqref{absbrezzi1}-\eqref{absbrezzi4}
  are satisfied. It can also be seen that this is not the case if $\Lambda=H_{00}^{-\half}(\Gamma)$.
  The bound \eqref{eq:stokes_estimate} is verified in Figure \ref{fig:stokesreg}.
\end{example}

\begin{example}[Demonstration of Assumption \ref{assumption:darcyreg}]\label{ex:darcyreg}
  Let $\Omega_p=\left[0, 1\right]^2$  with $\Gamma=\set{(x, y)\in\partial{\Omega_p}\,|\,x=0}$.
  $\partial\Omega_{p, D}=\set{(x, y)\in\partial{\Omega_p}\,|\,x=1}$.
  As in Example \ref{ex:stokesreg} the Assumption \ref{assumption:darcyreg} is demonstrated
  via the spectrum of the discrete preconditioned problem $\mathcal{A}x=\beta\mathcal{B}_{00}^{-1}x$ where
  $\mathcal{A}$ induced the bilinear form in \eqref{eq:bab_darcy_op} and $\mathcal{B}_{00}$
  is the Riesz map preconditioner with respect to the norms of $\mathbf{W}$, i.e.
  \begin{equation}\label{eq:darcy_example}
    \mathcal{A} = \begin{pmatrix}
        K^{-1}I & -\nabla & T_n\dual\\
        \nabla\cdot & &\\
        T & &\\
    \end{pmatrix},
    \,
    \mathcal{B}_{00} = \begin{pmatrix}
    K^{-1}(I - \nablab\nablab\cdot) &  &\\
    & K I &\\
    & & K{(-\Delta + I)}_{00}^{1/2}
\end{pmatrix}^{-1}.    
  \end{equation}
  Operator $\mathcal{B}$ then differs from $\mathcal{B}_{00}$ by using $K(-\Delta + I)^{1/2}$
  in the multiplier block.\\
  \begin{minipage}{0.95\textwidth}
    \begin{minipage}{0.49\textwidth}
      \centering
      \footnotesize{
      \begin{tabular}{c|ccc||c}
        \hline
        \multirow{2}{*}{$h$} & \multicolumn{3}{c||}{$\mathcal{B}_{00}\mathcal{A}$} & {$\mathcal{B}\mathcal{A}$}\\
        \cline{2-5}
        & $K=1$ & $10^{-4}$ & $10^{-8}$ & $1$\\
        \hline
$2^{-1}$ & 3.47 & 3.47 & 3.47 & 4.92\\
$2^{-2}$ & 3.52 & 3.52 & 3.52 & 5.55\\
$2^{-3}$ & 3.53 & 3.53 & 3.53 & 6.14\\
$2^{-4}$ & 3.54 & 3.54 & 3.54 & 6.71\\
$2^{-5}$ & 3.54 & 3.54 & 3.54 & 7.28\\
$2^{-6}$ & 3.54 & 3.54 & 3.54 & 7.86\\
$2^{-7}$ & 3.54 & 3.54 & 3.54 & 8.43\\
\hline
      \end{tabular}
        }
      \captionof{table}{Spectral condition numbers of preconditioned Darcy problem \eqref{eq:bab_darcy_op}. $\mathcal{B}_{00}$ is robust in $h$ and $K$.
        Results for $\mathcal{B}$ show that $H^{1/2}$ is not suitable if
        $\Gamma\cap\partial\Omega_{p, D}=\emptyset$.        
  }
  \label{tab:darcyreg}  
    \end{minipage}
    \begin{minipage}{0.49\textwidth}
      \begin{figure}[H]
      \centering
      \includegraphics[width=\textwidth]{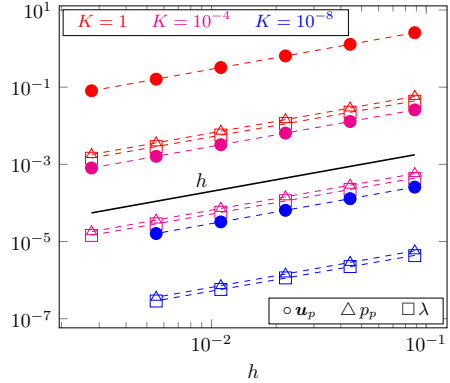}
      \vspace{-25pt}      
      \caption{Approximation errors of Darcy problem \eqref{eq:bab_darcy_op}
        measured in norm due to $\mathcal{B}_{00}^{-1}$. Discretization by \Delm elements.}
        \label{fig:darcyreg}        
      \end{figure}
    \end{minipage}
  \end{minipage}
\vspace{0.2cm}

  Table \ref{tab:darcyreg} shows the condition numbers
  of the preconditioned Darcy problem discretized by \Delm\   with different values of $K$. We observe that the condition
  numbers are practically constant with respect to $K$ and $h$. Moreover, the fact that $\mathcal{B}$
  leads to unbounded spectra shows that Darcy problem \eqref{eq:bab_darcy_op} is not well-posed
  with $\Lambda=H^{1/2}(\Gamma)$. Finally, the estimate \eqref{eq:darcy_estimate} is verified in
  Figure \ref{fig:darcyreg}.
\end{example}

We remark the quadratic, respectively linear convergence for the velocities and pressures
in the Stokes and Darcy subproblems, cf. Figure \ref{fig:stokesreg}, \ref{fig:darcyreg}, is
in agreement with the well-known theory for the approximation by Taylor-Hood and stable mixed-Poisson
elements. The stability of the rates with respect to parameter variations then provides
evidence for estimates \eqref{eq:stokes_estimate} and \eqref{eq:darcy_estimate}.

Following Examples \ref{ex:stokesreg} and \ref{ex:darcyreg} the well-posedness
of the coupled Darcy-Stokes problem is proved in Theorem \ref{thm:DSwellposedness}.
\begin{theorem} \label{thm:DSwellposedness}
  Let $\partial \Omega_i = \Gamma \cup \partial \Omega_{i, D} \cup \partial \Omega_{i, N}$, $i=p, f$
  such that $\semi{\Omega_{i, N}} > 0$, $\semi{\Omega_{i, D}} > 0$ and $\Gamma\cap\partial\Omega_{i, D}=\emptyset$.
  Further let
  \begin{equation*}
    \begin{split}
      \mathbf{W} = &\smuH \times \isKHd \times \ismuL \\
                   &\times \sKL \times \left(\tfrac{1}{\sqrt{\mu}}H^{-\half}(\Gamma)\cap \sqrt{K}H_{00}^{\half}(\Gamma)\right).
    \end{split}
  \end{equation*}
Then if Assumptions \ref{assumption:stokesreg} and \ref{assumption:darcyreg} hold, the Darcy-Stokes operator $\AA$ in \eqref{coeff:darcy:stokes} is an isomorphism mapping 
$\mathbf{W}$ to $\mathbf{W}'$ such that 
$\|\AA\|_{\mathcal{L}(\mathbf{W},\mathbf{W}')} \le C$ and 
$\|\AA^{-1}\|_{\mathcal{L}(\mathbf{W}',\mathbf{W})} \le C^{-1}$ 
where $C$ is independent of $\mu$, $K$, and $D$.
\end{theorem}

\begin{proof}[Proof of Theorem \ref{thm:DSwellposedness}] 
  By Assumption \ref{assumption:stokesreg} and \ref{assumption:darcyreg} the
  Brezzi conditions \eqref{absbrezzi1}-\eqref{absbrezzi2} hold for the Stokes and
  Darcy subproblems. In order to apply Theorem \ref{thm:abstractcoupledprob}, it remains to show
  the coercivity conditions \eqref{absbrezzi1kerB}. First, consider the Stokes subproblem.
  Because $\semi{\partial\Omega_{f, D}} > 0$, by the Poincar{\'e} inequality there exists a
  constant $C_f > 0$ depending only on the domain so that
  $\|\nabla \uf \|^2_{L^2(\Omega_f)} \geq C_f \|\uf \|^2_{\la{H}^1(\Omega_f)}$ for any $\uf \in \la{H}^1_{0, D}(\Omega_f)$. Hence
  \begin{align*}
    \mu (\nabla \uf, \nabla \uf) + (\tG\uf, \tG\uf)_{\Gamma} &= \mu \|\nabla \uf \|^2_{\la{L}^2(\Omega_f)} + D \|\uf \|^2_{\la{L}_{t}^2(\Gamma)} \\
                                                           & \geq \left (C_f\mu \|\uf \|^2_{\la{H}^1(\Omega_f)} + D \|\uf \|^2_{\la{L}_{t}^2(\Gamma)}  \right ) \\
    & = \min{(1, C_f)}\left (\|\uf \|^2_{\sqrt{\mu} \la{H}^1(\Omega_f)} + \|\uf \|^2_{\sqrt{D}\la{L}_{t}^2(\Gamma)}  \right ) \\
    & = \min{(1, C_f)} \|\uf\|^2_{\smuH}.
  \end{align*}
  The coercivity condition thus holds with constant $\min{(1, C_f)}$. Next,
  consider the Darcy subproblem. Here the coercivity on $\ker B_p$ follows by definition.
  Indeed for any $\up\in\ker B_p$ we have $\nabla\cdot\up = 0$ and 
  \[
  K^{-1}(\up, \up) = \|\up\|^2_{\tfrac{1}{\sqrt{K}} \la{L}^2(\Omega_p)} =\|\up\|^2_{\tfrac{1}{\sqrt{K}}\la{H}_{0, D}(\text{div}, \Omega_p)}.
  \]
  
Hence we have established condition \eqref{absbrezzi1kerB}. The assumptions of Theorem \ref{thm:abstractcoupledprob} are thus all satisfied, showing well-posedness as desired. 
\end{proof}



\begin{remark}[Homogeneous Dirichlet conditions]\label{rmrk:DS_DD}
  In Example \ref{ex:darcy_right}, Theorem \ref{thm:DSwellposedness} we showed that the preconditioner was robust in
  case where only the Neumann boundaries of both problems are intersected by the interface.
On the other hand, the Darcy-Stokes problem with $\semi{\partial\Omega_{i, N}}=0$ has been shown
well-posed by \cite{layton2002coupling,galvis2007non}. From the point of
view of abstract Theorem \ref{thm:abstractcoupledprob} the case is thus
interesting as the analogues of Assumptions \ref{assumption:stokesreg} and \ref{assumption:darcyreg}
have been established in~\cite{galvis2007non}.

With the structure of the multiplier space hinted at in Examples \ref{ex:poisson_bcs}, \ref{ex:stokesreg}
and \ref{ex:darcyreg} the difficulty of the homogeneous Dirichlet conditions is the
fact that the operator \eqref{coeff:darcy:stokes} is singular with a kernel
$z=(\boldsymbol{0}, \boldsymbol{0}, 1, 1, 1)$. Let for
simplicity all material parameters be unity, so that
\[
\mathbf{W} = \mathbf{H}^1_{0,D}(\Omega_f)\times \HdivD{D}{\Omega_p} \times L^2(\Omega_f)\times L^2(\Omega_p)\times H^{1/2}(\Gamma),
\]
and assume a compatible right-hand side, i.e. $(1, f_p)_{\Omega_p}=0$. Then \cite{galvis2007non} proves
well-posedness of \eqref{eq:darcy_stokes_weak} in $\mathbf{W}^{\perp}=\set{w\in \mathbf{W}\,|\,(w, z)=0}$.
However, this constraint is impractical in our setting as it enforces additional
structure on the multiplier space namely, $H^{\half}\cap H_{00}^{-\half}\cap L^2/\reals$.
Instead, we have found it convenient to normalize the Stokes pressure. That is, first a solution with 
$p_f\in L^2(\Omega_f)/\reals$ is found and afterwards we renormalize as
$p_f = p_f - (1, p_p)_{\Omega_f} - (1, \lambda)_{\Gamma}$.
The first step is thus a Darcy-Stokes problem formulated in
\begin{equation}\label{eq:DS_DD}
  \begin{split}
    \mathbf{W} = \smuH \times \isKHd \times \ismuL \\
    \times \sKL \times \left(\tfrac{1}{\sqrt{\mu}}H_{00}^{-\half}(\Gamma)\cap \sqrt{K}H^{\half}(\Gamma)\right)\times\mu\reals,
  \end{split}
\end{equation}  
where the additional unknown enforces $p_f\in L^2(\Omega_f)/\reals$. We remark
that its $\mu$ scaling is necessary for parameter independence.

The Riesz map preconditioner of \eqref{eq:DS_DD} does not interfere with the Lagrange multiplier
and can be implemented using the same solvers as \eqref{eq:B_darcy_stokes}. Experiments
demonstrating robustness of the preconditioner are summarized in Figure \ref{fig:DS_DD}.
\begin{figure}
  \centering
  \includegraphics[height=0.5\textwidth]{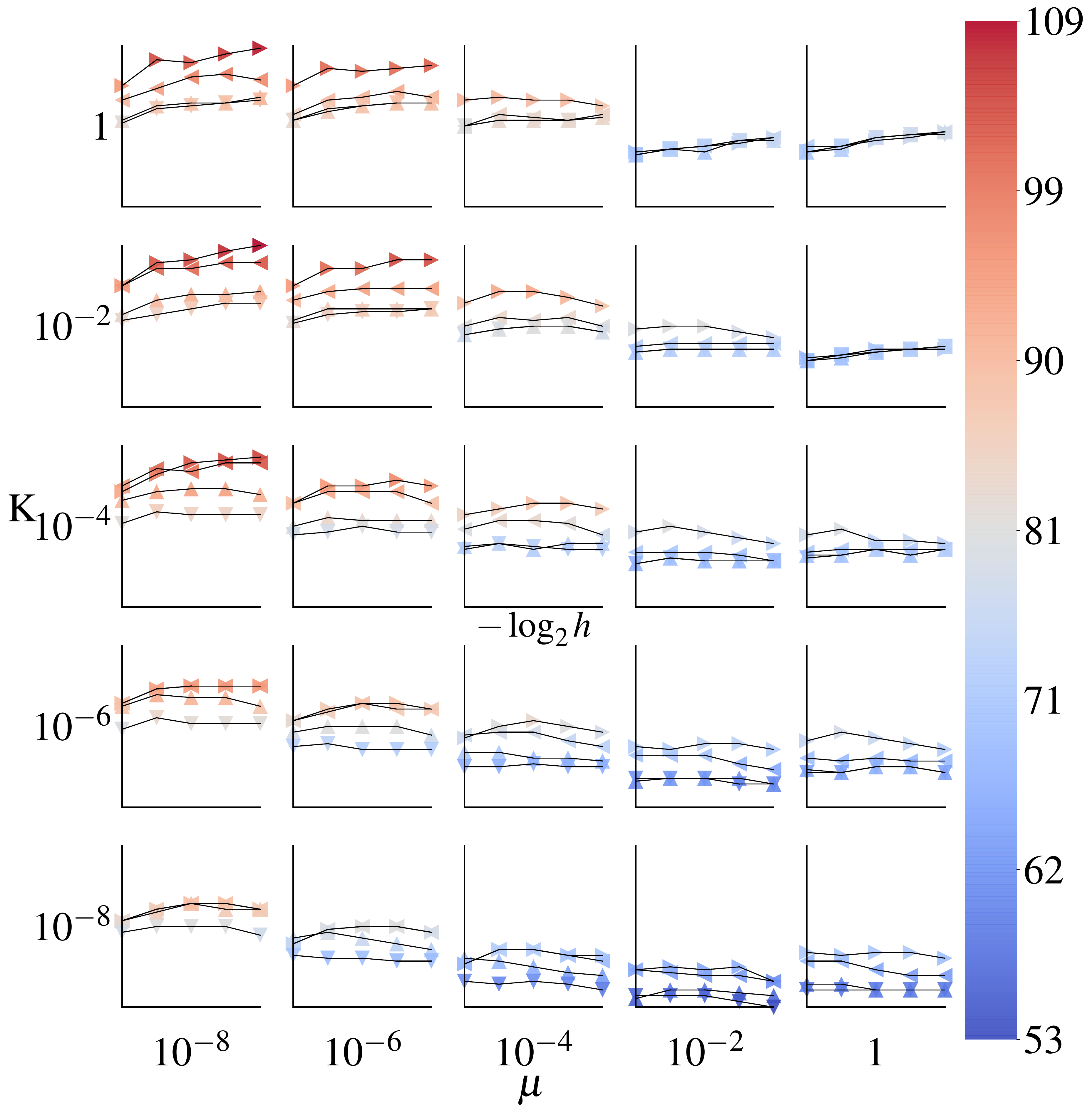}
  \includegraphics[height=0.5\textwidth]{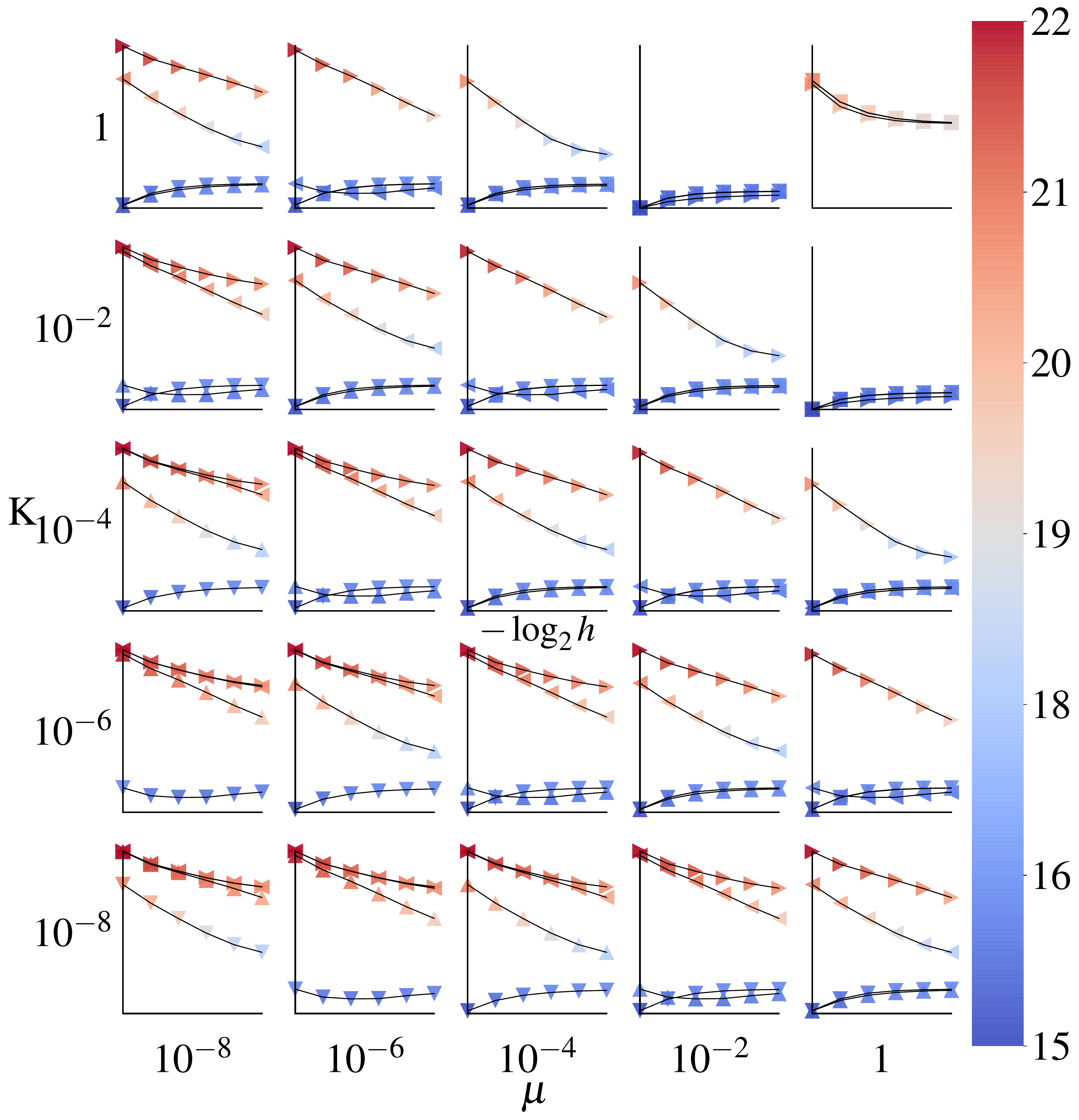}
  \vspace{-15pt}
  \caption{Darcy-Stokes problem with homogeneous Dirichlet boundary conditions
    preconditioned by Riesz map preconditioner of \eqref{eq:DS_DD}.
(Left) Number of preconditioned MinRes iterations.
(Right) Spectral condition number. 
The coarsest mesh for left plot has
$h=2^{-3}$ while $h=2^{-1}$ in the right plot. For fixed $K$, $\mu$ subplots the horizontal axis is scaled as $-log_2 h$ so that the system size grows from left to right. Values of $\alpha_{\text{BJS}}=10^{-6}, 10^{-4}, 10^{-2}, 1$ 
    are encoded with markers $\triangledown$, $\triangle$, $\triangleleft$, $\triangleright$.
}
 \label{fig:DS_DD}
\end{figure}
\end{remark}

\section{Robust preconditioners for the Stokes-Navier system}\label{sec:stokes-elasticity}
Let $\Omega_f$, $\Omega_p$ be as in \S\ref{sec:intro}. We consider a model problem describing the
interaction of a viscous fluid occupying domain $\Omega_f$ with a linear solid $\Omega_p$ undergoing
small elastic deformations. Let $L_i$, $i=p,f, I$ be the Lagrangians
\[
\begin{aligned}
  L_f &= (2\mu \Epsilon(\uf), \Epsilon(\uf))_{\Omega_f} -(p_f, \nabla\cdot\uf)_{\Omega_f} -(\boldsymbol{f}_f, \uf)_{\Omega_f} - (\boldsymbol{h}_f, \uf)_{\partial\Omega_{f, N}},\\
  L_p &= (2\nu \Epsilon(\up), \Epsilon(\up))_{\Omega_p} -(p_p, \nabla\cdot\up)_{\Omega_p} - (\eta^{-1}p_p, p_p)_{\Omega_p}-(\boldsymbol{f}_p, \up)_{\Omega_p}\\
      & -(\boldsymbol{h}_p, \up)_{\partial\Omega_{p, N}},\\
  L_I &= (\boldsymbol{\lambda}, k\uf - \up - \boldsymbol{g})_{\Gamma},
\end{aligned}
\]
on the respective subdomains. Here, $\Epsilon(\boldsymbol{v})=\tfrac{1}{2}(\nabla\boldsymbol{v}+\nabla\boldsymbol{v}^T)$
and the material parameters of the model are fluid viscosity $\mu$ and Lam{\' e} constants $\nu$, $\eta$.
The Lagrangian of the coupled problem then reads $L=L_p+L_f+L_I$. 

The coupling between the Stokes and the Navier equations consists
of two conditions. The balance of normal stress 
\[
    \bsigma_f(\uf, p_f)\cdot \nG - \bsigma_p(\up, p_p)\cdot \nG = 0 \text{ on }\Gamma
\]
with $\bsigma_f(\uf, p_f) = 2\mu\Epsilon(\uf) -p_f\la{I}$, 
$\bsigma_p(\up, p_p) = 2\nu\Epsilon(\up) + p_p\la{I}$ is enforced weakly
 while a (vector valued) Lagrange multiplier enforces continuity of motion, i.e.   
$0=\tfrac{\partial L}{\partial\boldsymbol{\lambda}}$. The condition 
is a temporal discretization of the kinematic interface condition
stating that the fluid and solid velocity on $\Gamma$ should be equal,
c.f. e.g. \cite{gerstenberger2008extended}. Note that in general $\boldsymbol{g}_{\Gamma}\neq \mathbf{0}$ 
because $\boldsymbol{u}_p$ represents a displacement in the solid domain $\la{u}_f$ is a velocity in the fluid domain. Hence, 
$0=\tfrac{\partial L}{\partial\boldsymbol{\lambda}}$ expresses time-stepping in the solid domain and 
$\boldsymbol{g}$ will be the displacement on the previous time-step. Therefore, 
 $\boldsymbol{u}|_{\Omega_i}=\boldsymbol{u}_i$, $i=p, f$ is not continuous on $\Gamma$.
Finally, the boundary terms in the Lagrangian reflect Neumann boundary conditions while on the
remaining part Dirichlet data are assumed:
\[
\uf = \boldsymbol{u}^0_f\mbox{ on }\partial\Omega_{f, D},\quad
\up = \boldsymbol{u}^0_p\mbox{ on }\partial\Omega_{p, D}
\]
and
\[
\bsigma_f(\uf, p_f)\cdot\nG = \boldsymbol{h}_f\mbox{ on }\partial\Omega_{f, N},\quad
\bsigma_p(\up, p_p)\cdot\nG = \boldsymbol{h}_p\mbox{ on }\partial\Omega_{p, N}.
\]

We remark that in $L_p$ we consider the mixed-formulation of linear elasticity as
our main interest is in the nearly incompressible regime 
(that is $\eta\to\infty$) in which the displacement formulation is known to suffer
from locking \cite{klawonn1998block}. In the following we derive preconditioners
for the coupled problem for the case $\eta\gg 1$, $k\leq 1$ and $\mu>0$, $\nu>0$. However,
we shall focus on the more challenging case $0 < \mu \leq 1$. Further, as $\nu$ is practical
for rescaling, e.g \cite{hong2017parameter}, we let $\nu=1$ and 
only the robustness with respect to $k$, $\mu$ and $\eta$ shall be addressed further.

Letting
  $\mathbf{W}=\la{H}^1_{0, D}(\Omega_f)\times \la{H}^1_{0, D}(\Omega_p)\times L^2{(\Omega_f)}\times L^2{(\Omega_p)}\times \la{H}^{-1/2}(\Gamma)$
the extremal points $(\uf, \up, p_f, p_p, \boldsymbol{\lambda})\in \la{W}$ of $L$ satisfy
\begin{equation}\label{eq:elast_stokes_weak}
\begin{aligned}
a((\uf, \up), (\vf, \vp)) + b((\vf, \vp), (p_f, p_p, \boldsymbol{\lambda})) &= f((\vf, \vp)),\\
b((\uf, \up), (q_f, q_p, \boldsymbol{w})) -\eta^{-1} (p_p, q_p)_{\Omega_p} &= g((q_f, q_p, \boldsymbol{w}))
\end{aligned}
\end{equation}
for all $(\vf, \vp)\in \mathbf{H}^1_{0, D}(\Omega_f)\times \mathbf{H}^1_{0, D}(\Omega_p)$
and all $(q_f, q_p, \boldsymbol{w})\in L^2{(\Omega_f)}\times L^2{(\Omega_p)}\times \la{H}^{-1/2}(\Gamma)$.
Here the bilinear forms $a$, $b$ are defined as 
\[
\begin{aligned}
  a((\uf, \up), (\vf, \vp)) &= 2\mu  (\Epsilon(\uf), \Epsilon(\vf))_{\Omega_f} + 2 (\Epsilon(\up), \Epsilon(\vp))_{\Omega_p},\\
  b((\uf, \up), (q_f, q_p, \boldsymbol{w})) &= (q_f, \nabla\cdot\uf)_{\Omega_f} + (q_p, \nabla\cdot\up)_{\Omega_p} + (\boldsymbol{w}, k T\uf-T\up)_{\Gamma}.
\end{aligned}
\]
while
\[
f((\vf, \vp)) = \sum_{i=p, f}(\boldsymbol{f}_i, \boldsymbol{v}_i)_{\Omega_i} + (\boldsymbol{h}_i, \boldsymbol{v}_i)_{\partial\Omega_{i, N}}
\quad\mbox{ and }\quad
g((q_f, q_p, \boldsymbol{w})) = (\boldsymbol{g}, \boldsymbol{w}).
\]
We remark that the trace operators above act on vector fields.

Problem \eqref{eq:elast_stokes_weak} can be equivalently stated in terms of an
operator $\mathcal{A}:\la{W}\rightarrow \la{W}\dual$
\begin{equation}
\label{eq:elast_stokes_A}
\AA =    \left( \begin{array}{cc|ccc}
                  - 2\mu\nablab\cdot {\Epsilon} &  & -\nabla & & k T\dual \\ 
                              &-2\nablab\cdot {\Epsilon} &  &  -\nabla & -T\dual\\ \hline
                  \nabla \cdot&  &  &  \\ 
                              & \nabla \cdot  &  & -{\eta}^{-1} I \\ 
                  kT & -  T &  &   
     \end{array} \right).
\end{equation}
Observe that compared to the abstract problem
\eqref{eq:abstractcoupledprob} operator \eqref{eq:elast_stokes_A} has an
additional term on the diagonal, cf. $-{\eta}^{-1} I$, and the interface coupling
contains an explicit parameter. Considering the case where $\semi{\partial\Omega_{i, D}} > 0$
and $\Gamma\cap \partial\Omega_{i, D}=\emptyset$ we aim to show that the operator
\begin{equation}\label{eq:B_stokes_navier}
  \mathcal{B}=\begin{pmatrix}
-2\mu\nablab\cdot {\Epsilon} & & & &\\
&-2\nablab\cdot {\Epsilon} & & &\\
 & & \frac{1}{\mu}I & &\\
 & & & I & \\
& &  & &(\frac{k^2}{\mu}+1)(-{\Delta+I})^{-1/2}\\
  \end{pmatrix}^{-1}
\end{equation}
defines a parameter robust preconditioner for the Stokes-Navier system.
Note that the fractional operator is vector valued.

Due to the penalty term $-{\eta}^{-1} I$ in \eqref{eq:elast_stokes_A},
robustness of the preconditioner \eqref{eq:B_stokes_navier} 
does not follow directly from Theorem \ref{thm:abstractcoupledprob}. However, the abstract
framework will be used to show hypothesis \eqref{eq:penalty_brezzi} of the following result due to
\cite{braess1996stability}.

\begin{theorem}[\cite{braess1996stability}]\label{thm:penalty}
  Let $V, Q$ be Hilbert spaces and $A:V\rightarrow V\dual$, $B:V\rightarrow Q\dual$,
  $C:Q\rightarrow Q\dual$ be such that
  \begin{subequations}
    \begin{align}
      \label{eq:penalty_brezzi}
    &\begin{pmatrix}
      A & B\dual\\
      B &
     \end{pmatrix}\text{ satisfies the Brezzi conditions},\\
      \label{eq:penalty_C}
      & C \text{ is bounded positive semidefinite on $Q$},\\
      \label{eq:penalty_A}
      & A \text{ is positive semidefinite on $V$}.
  \end{align}
  \end{subequations}
  Then $\mathcal{A}=\begin{pmatrix}
      A & B\dual\\
      B & -t^2C
  \end{pmatrix}$ is an isomorphism $W\rightarrow W\dual$ and $\mathcal{A}^{-1}$ is
  uniformly bounded for $0\leq t \leq 1$.
\end{theorem}

In order to apply Theorem \ref{thm:abstractcoupledprob} to verify the Brezzi
conditions \eqref{eq:penalty_brezzi} the individual Stokes/Navier subproblems must 
satisfy \eqref{absbrezzi1}-\eqref{absbrezzi2} and the estimate \eqref{wellposednessdfn}. 
Here, we shall assume this result and later support it by numerical experiments similar 
to Assumptions \ref{assumption:stokesreg}, \ref{assumption:darcyreg}. 

\begin{assumption}[Navier subproblem]\label{assume:navier}
  Let $\Omega_f\subset\reals^2$ be a bounded domain with boundary decomposition
  $\partial\Omega_f=\Gamma\cup\partial\Omega_{f, D}\cup\partial\Omega_{p, N}$ where the
  components are assumed to be of non-zero measure and $\partial\Omega_{f, D}\cap\Gamma=\emptyset$. 
  We consider the problem
  \[
    \begin{aligned}
  -\nabla\cdot(\bsigma_f(\uf, p_f)) &= \ff&\mbox{ in }\Omega_f,\\
  \nabla\cdot\uf &= 0  &\mbox{ in }\Omega_f\cup\Gamma,\\
  \uf &= \boldsymbol{u}_f^0&\mbox{ on }\partial\Omega_{f, D},\\
  \nG\cdot\bsigma &= \boldsymbol{h}_f&\mbox{ pm }\partial\Omega_{f, N}.
    \end{aligned}
  \]

  Let $\mathbf{V}=\sqrt{\mu}\mathbf{H}^1_{0, D}(\Omega_f)$ $Q= \tfrac{1}{\sqrt{\mu}}L^2(\Omega_f)$, $\mathbf{\Lambda}=\tfrac{1}{\sqrt{\mu}} \la{H}^{-1/2}(\Gamma)$.
  For $\mathbf{W}=\mathbf{V}\times Q\times \mathbf{\Lambda}$ we define $a:\mathbf{W}\times \mathbf{W}\rightarrow\reals$, $L:\mathbf{W}\rightarrow\reals$ as
  
  \begin{equation}\label{eq:bab_navier_op}
    \begin{aligned}
      a((\uf, p_f, \boldsymbol{\lambda}), (\vf, q_f, \boldsymbol{w})) =& 2\mu (\Epsilon(\uf), \Epsilon(\vf)) + (p_f, \nabla \cdot \vf)  + (\nabla \cdot \uf, q_f)\\
      &+(\boldsymbol{\lambda}, T\vf)_{\Gamma} + (\boldsymbol{w}, T\uf)_{\Gamma},\\
L((\vp, q_p, \boldsymbol{w})) =& (\ff, \vf) + (\boldsymbol{v}_f, \boldsymbol{w})_{\Gamma}.
    \end{aligned}
\end{equation}  
  Then the problem: Find $(\uf, p_f, \boldsymbol{\lambda})\in \la{W}$ such that
  \[
  a((\up, p_p, \boldsymbol{\lambda}), (\vp, q_p, \boldsymbol{w})) = L((\vp, q_p, \boldsymbol{w})),\quad\forall(\vp, q_p, \boldsymbol{w})\in \mathbf{W}
  \]
  has a unique solution which satisfies 
  \begin{equation}\label{eq:navier_estimate}
  \|(\up, p_p, \boldsymbol{\lambda})\|_{\mathbf{W}} \leq C\left ( \|\ff\|_{\mathbf{V}\dual} +  \|\boldsymbol{g}\|_{\mathbf{\Lambda}\dual} \right )
  \end{equation}
  with $C$ independent of $\mu$.
\end{assumption}

We remark that \eqref{eq:bab_navier_op} differs from \eqref{eq:bab_stokes_op} by using the
full (vector) trace operator. The Navier problem with the normal trace operator, i.e. $\uf\cdot\nG$ 
enforced by Lagrange multiplier, which is a scalar in case $\Omega_f\subset\reals^2$, was shown to 
be well-posed in \cite{bertoluzza2017boundary}.

\begin{example}[Demonstration of Assumption \ref{assume:navier}]\label{ex:navier}
  Let $\Omega_f=\left[0, 1\right]^2$  with $\Gamma=\set{(x, y)\in\partial{\Omega_f}\,|\,x=0}$ and $\partial\Omega_{f, D}=\set{(x, y)\in\partial{\Omega_f}\,|\,x=1}$. We demonstrate that
      Assumption \ref{assume:navier} holds by considering the spectra of the
      preconditioned problem $\mathcal{A}x=\beta\mathcal{B}^{-1}x$ where $\mathcal{A}$
      is the operator due to the bilinear form in \eqref{eq:bab_navier_op} and $\mathcal{B}$
      is the Riesz map preconditioner induced by the space $\la{W}$, i.e.
\begin{equation}\label{eq:navier_example}
    \mathcal{A} = \begin{pmatrix}
        -2\mu\nablab\cdot {\Epsilon} & -\nabla & T\dual\\
        \nabla\cdot & &\\
        T & &\\
    \end{pmatrix},
    \,
\mathcal{B} = \begin{pmatrix}
      -2\mu\nablab\cdot\Epsilon &  &\\
    & \mu^{-1}I &\\
    & & \mu^{-1}{({-\Delta + I})}^{-1/2}
    \end{pmatrix}^{-1}.    
\end{equation}
As in this example newly the trace is a vector valued operator we define,
in addition to $\mathcal{B}$, the preconditioners $\mathcal{B}_{00}$
$\mathcal{B}_{n0}$, $\mathcal{B}_{t0}$. In $\mathcal{B}$ both the normal and
the tangential component of the multiplier are considered in $H^{-1/2}$. In
the remaining preconditioners both, respecively normal and tangential components
are assumed in $H^{-1/2}_{00}$

\begin{minipage}{0.45\textwidth}
      \centering
      \footnotesize{
      \begin{tabular}{c|ccc}
        \hline
            \multirow{2}{*}{$h$} & \multicolumn{3}{c}{$\mathcal{B}\mathcal{A}$}\\
        \cline{2-4}    
        & $\mu=1$ & $10^{-4}$ & $10^{-8}$\\
        \hline
$2^{-2}$ & 25.60 & 25.60 & 25.60 \\
$2^{-3}$ & 26.30 & 26.30 & 26.30 \\
$2^{-4}$ & 26.78 & 26.78 & 26.78 \\
$2^{-5}$ & 26.91 & 26.91 & 26.91 \\
$2^{-6}$ & 26.95 & 26.95 & 26.95 \\
        \hline
        \hline
        $h$ & {$\mathcal{B}_{00}\mathcal{A}$} & {$\mathcal{B}_{n0}\mathcal{A}$} & {$\mathcal{B}_{t0}\mathcal{A}$}\\
\hline        
$2^{-2}$ &39.04 &27.24 &38.36 \\
$2^{-3}$ &46.00 &30.93 &44.10 \\
$2^{-4}$ &51.40 &33.78 &48.77 \\
$2^{-5}$ &56.28 &36.52 &52.89 \\
$2^{-6}$ &60.97 &39.36 &56.81 \\        
\hline
      \end{tabular}
        }
      \captionof{table}{
        Spectral condition numbers of preconditioned problem \eqref{eq:bab_navier_op}. $\mathcal{B}$ is robust in $h$ and $\mu$.
        Results with the remaining preconditioners use $\mu=1$ and suggest that well-posedness
        requires both multiplier components in $H^{-1/2}$.
  }
  \label{tab:navier}  
    \end{minipage}
\begin{minipage}{0.49\textwidth}
    \begin{figure}[H]
      \centering
      \includegraphics[width=\textwidth]{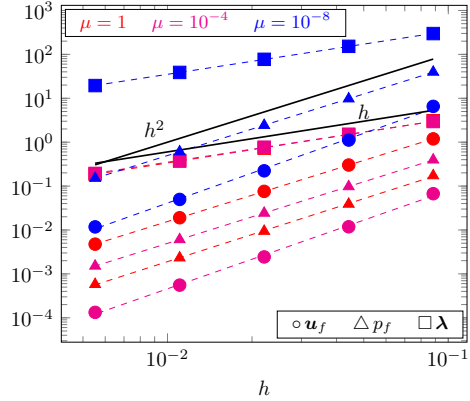}
      \vspace{-25pt}      
      \caption{
        Approximation errors of Navier problem \eqref{eq:bab_navier_op}
        measured in norm due to $\mathcal{B}^{-1}$. Discretization by \vecSelm elements.
      }
        \label{fig:navier}
    \end{figure}
\end{minipage}

  With preconditioner based on $\la{\Lambda}=\la{H}^{-1/2}$
  the condition numbers of \eqref{eq:navier_example} shown in Table \ref{tab:navier} appear bounded as the mesh is refined
  and are practically independent of $\mu$. With preconditioners based on $H^{-1/2}_{00}$
  for some of the multiplier components the results are unbounded.
  
  The estimate \eqref{eq:navier_estimate} is verified in Figure \ref{fig:navier}.
  We remark that \vecSelm\ elements were used in the example. Note also that the multiplier convergence
  is linear, cf. quadratic in Example \ref{ex:stokesreg}. However, in all the testcases
  the manufuctured Lagrange multiplier was a trigonometric function.
\end{example}

As in the case of Assumptions \ref{assumption:stokesreg} and \ref{assumption:darcyreg},
the numerical experiments in Example \ref{ex:navier} show that by using the norms of Assumption
\ref{assume:navier}, the preconditioned system has a condition number bounded in both discretization and material parameters, supporting Assumption
\ref{assume:navier}. We remark that the quadratic convergence observed in Figure \ref{fig:navier}
is an agreemenent with the analysis of \cite{bertoluzza2017boundary} where \eqref{eq:bab_navier_op}
was studied with the normal trace operator.


\begin{theorem} \label{thm:SEwellposedness}
  Let $\Omega_f$, $\Omega_p$ be such that $\semi{\partial\Omega_{i, D}} > 0$
  and $\partial\Omega_{i, D}\cap\Gamma=\emptyset$, $i=p, f$. Let
  \[
  \la{W}=
  \sqrt{\mu}\mathbf{H}^1_{0, D}(\Omega_f)\times
  \mathbf{H}^1_{0, D}(\Omega_p)\times
  \tfrac{1}{\sqrt{\mu}}L^2{(\Omega_f)}\times
  L^2(\Omega_p)\times 
  \sqrt{1+\frac{k^2}{\mu}} \mathbf{H}^{-1/2}(\Gamma).
  \]
Then if Assumption \ref{assume:navier} holds,
the operator $\AA$ in \eqref{eq:elast_stokes_A} is an isomorphism mapping 
$\la{W}$ to $\la{W}\dual$ such that 
$\|\AA\|_{\mathcal{L}(\la{W},\la{W}\dual)} \le C$ and 
$\|\AA^{-1}\|_{\mathcal{L}(\la{W}',\la{W})} \le \frac 1C$ 
where $C$ is independent of $\mu$, $k$, and $\eta$.
\end{theorem}
\begin{proof}
  Assuming Assumption \ref{assume:navier} holds, Theorem \ref{thm:abstractcoupledprob}
  verifies the condition \eqref{eq:penalty_brezzi}. Since $C=\diag(0, I, 0)$ in \eqref{eq:elast_stokes_A}
  boundedness and semi-definiteness of $C$ in \eqref{eq:penalty_C} are satisfied. It remains
  to verify the coercivity condition \eqref{eq:penalty_A} for $A=\diag(2\mu\nablab\cdot {\Epsilon}, 2\nablab\cdot {\Epsilon})$.
  Using Korn's inequality on both subdomains $i=f, p$, there exist $C_i>0$ such
  that $2\|\Epsilon(\boldsymbol{u}_i)\|^2_{\la{L}^2(\Omega_i)} \geq C_i \|\nabla \boldsymbol{u}_i\|^2_{\la{L}^2(\Omega_i)}$.
  Then
  \[
  \begin{aligned}
    \left (A(\uf, \up), (\uf, \up)\right ) &= 2 \mu\|\Epsilon(\uf)\|^2_{\la{L}^2(\Omega_f)} + 2 \|\Epsilon(\up)\|^2_{\la{L}^2(\Omega_p)}\\
    & \geq \min{\left(C_p, C_f\right)}\left(\mu \|\nabla \up\|^2_{\la{L}^2(\Omega_p)}+\|\up\|^2_{\la{L}^2(\Omega_p)}\right)\\
    & \geq C \left ( \mu \|\uf\|^2_{\la{H}^1_{0, D}(\Omega_f)} + \|\up\|^2_{\la{H}^1_{0, D}(\Omega_p)} \right ),
  \end{aligned}
  \]
  where the Poincar{\'e} inequality was
  used in the final step, cf. $\semi{\partial\Omega_{i, D}} > 0$. All asumptions of Theorem \ref{thm:penalty} are thus met.
\end{proof}

\begin{example}[Robust Stokes-Navier preconditioner]\label{ex:stokes_navier}
  We consider \eqref{eq:elast_stokes_weak} with the geometrical setup of Darcy-Stokes 
  Example \ref{ex:darcy_right}, see also Figure \ref{fig:DSdomains}. Using preconditioner 
  \eqref{eq:B_stokes_navier} and discretization in terms of \SNelm\ elements Figure 
  \ref{fig:NS_NN} shows the MinRes iterations counts and condition numbers. 
  Compared to the Darcy-Stokes problem the spread of the quantities is larger, however,
  both are bounded. Observe in particular that with the remaining parameters fixed
  the condition number is bounded in the 
  time stepping parameter $k$.

  \begin{figure}
    \centering
  \includegraphics[height=0.5\textwidth]{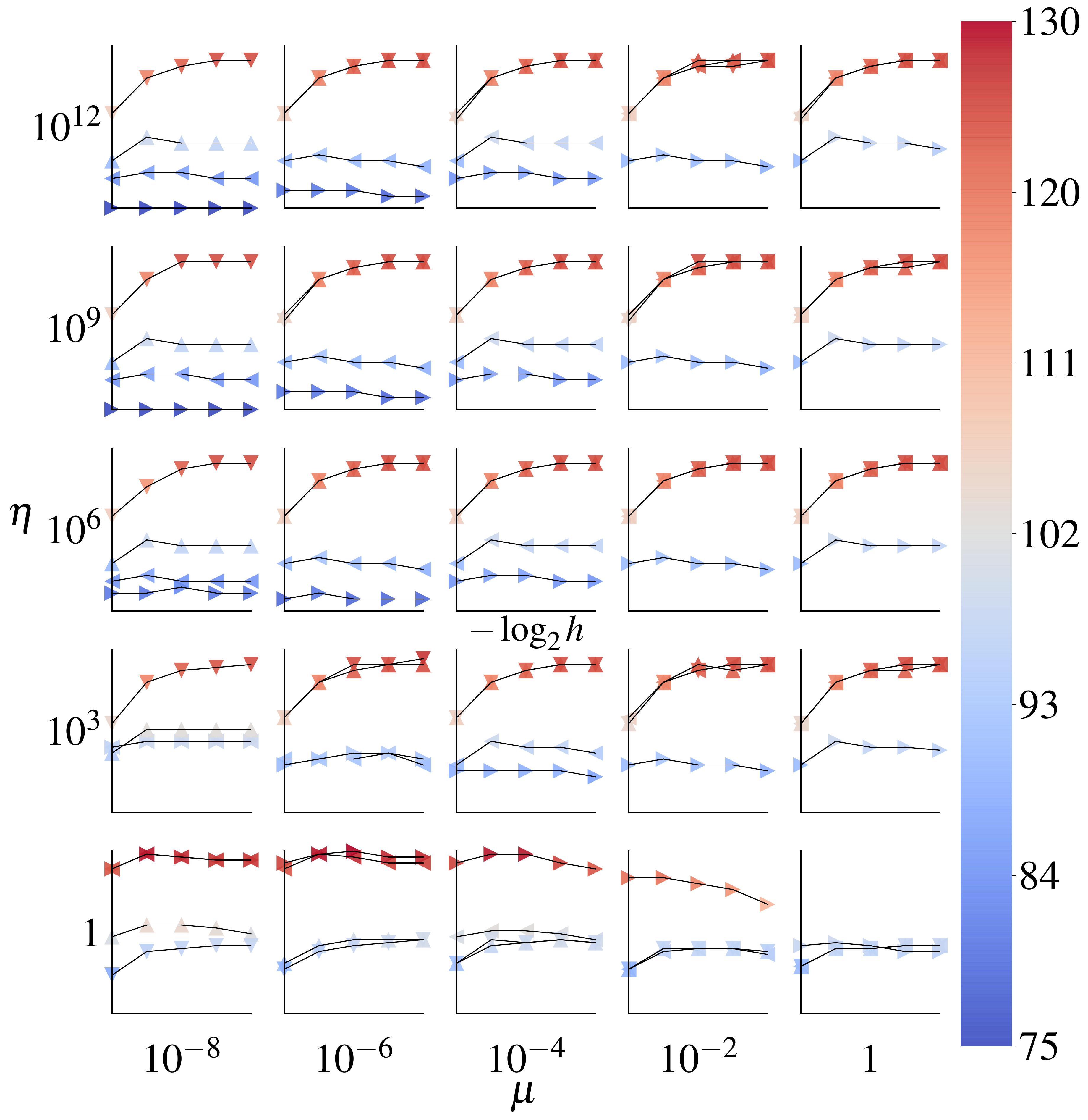}
  \includegraphics[height=0.5\textwidth]{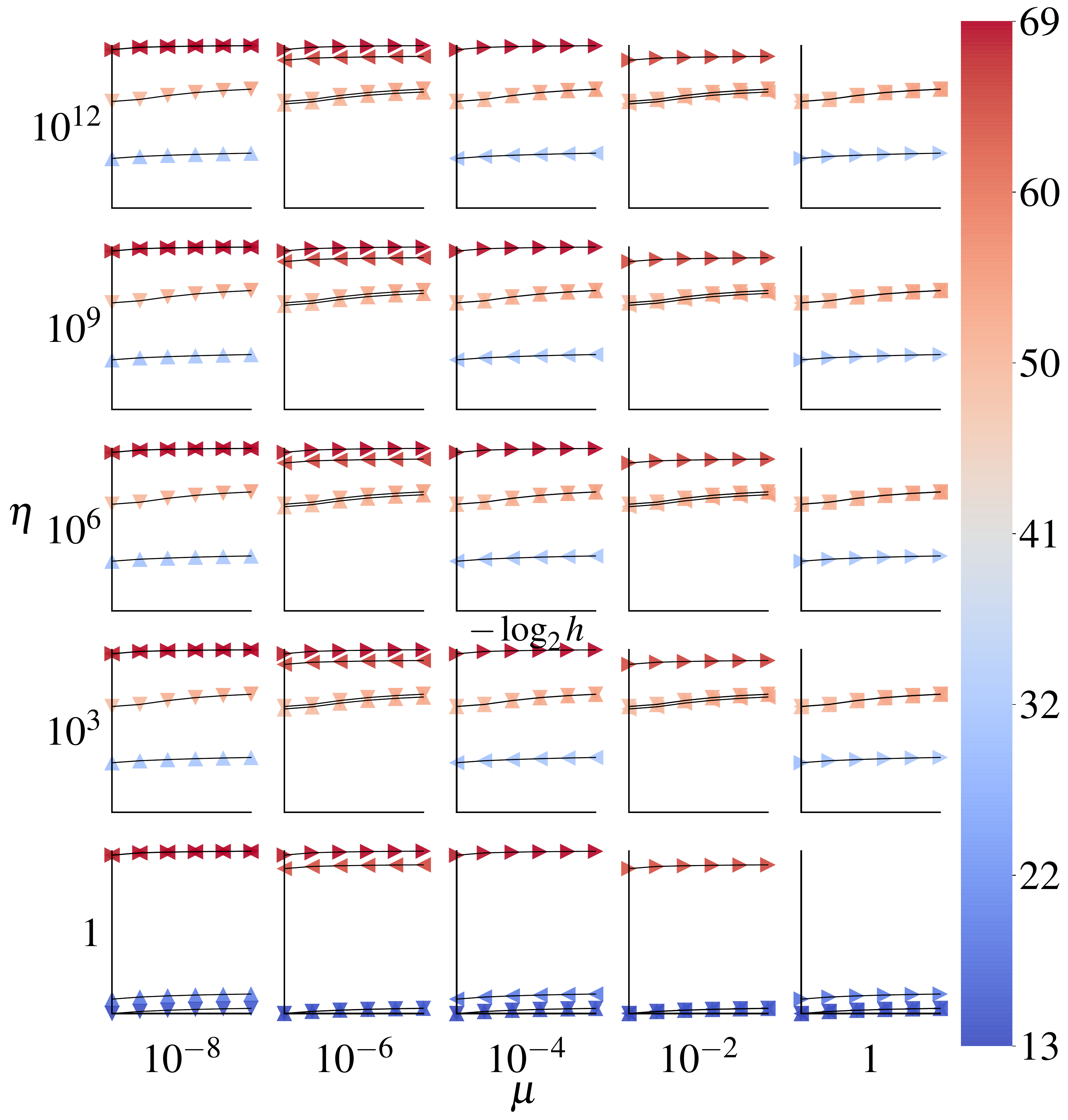}
    \vspace{-15pt}  
   \caption{
     Stokes-Navier problem with $\Gamma$ intersecting Neumann boundaries and $\semi{\partial\Omega_{i, D}} > 0$.
     Preconditioner \eqref{eq:B_stokes_navier} is used.
 (Left) Number of preconditioned MinRes iterations.
 (Right) Spectral condition number. 
 For fixed $\eta$, $\mu$ subplots the horizontal axis is scaled as $-log_2 h$ so that the system size grows from left to right. The coarsest mesh for left plot has $h=2^{-3}$ while $h=2^{-1}$ in the right plot. Values of $k=10^{-6}, 10^{-4}, 10^{-2}, 1$ 
    are encoded with markers $\triangledown$, $\triangle$, $\triangleleft$, $\triangleright$.
   }
  \label{fig:NS_NN}
 \end{figure}
\end{example}

We finally address the Stokes-Navier system equipped with homogeneous
Dirichlet conditions. In contrast to the Darcy-Stokes problem in Remark
\ref{rmrk:DS_DD}, the operator \eqref{eq:elast_stokes_A} in this case will not
be singular for $\eta < \infty$. However, the challenge  comes from the fact that the problem becomes 
singular, with a one-dimensional kernel,  in the incompressible limit and as such there is one vector that
is problematic. Our observations are summarized in Remark \ref{rmrk:SN_DD}. 

\begin{remark}[Homogeneous Dirichlet conditions]\label{rmrk:SN_DD} Let
  $\semi{\partial\Omega_{i, N}}=0$, $i=p, f$ in \eqref{eq:elast_stokes_weak}.
  As $\Gamma\cap\partial\Omega_{i, D}\neq\emptyset$ let, cf. Example \ref{ex:poisson_bcs}, 
  \begin{equation}\label{eq:SN_DD_space}
  \la{W}=
  \sqrt{\mu}\mathbf{H}^1_{0, D}(\Omega_f)\times
  \mathbf{H}^1_{0, D}(\Omega_p)\times
  \tfrac{1}{\sqrt{\mu}}L^2{(\Omega_f)}\times
  L^2(\Omega_p)\times 
  \sqrt{1+\frac{k^2}{\mu}} \mathbf{H}_{00}^{-1/2}(\Gamma).
  \end{equation}
  Considering $\mathcal{A}$ in \eqref{eq:elast_stokes_A} on $\la{W}$ the operator
  is non-singular, however, in the limit $\eta=\infty$, the vector $z=(\la{0}, \la{0}, k, 1, -\la{n})$
  forms the nullspace of $\mathcal{A}$. In turn the Brezzi
  conditions \eqref{eq:penalty_brezzi} of Theorem \ref{thm:penalty} do not hold independently of $\eta$
  on $\la{W}$.

  Using the Riesz map preconditioner based on $\la{W}$ we 
  illustrate below the  
  the sensitivity of the condition numbers to variations in $\eta$. Here the remaining
  parameters are fixed at 1. However, the lack of $\eta$-robustness was observed also if $0<\mu<1$ and $0<k<1$.
\begin{center}
  \footnotesize{
  \begin{tabular}{c|llll}
    \hline
    \multirow{2}{*}{$\eta$} & \multicolumn{4}{c}{$h$}\\
    \cline{2-5}
    & $2^{-1}$ & $2^{-2}$ & $2^{-3}$ & $2^{-4}$\\
\hline
1 & 26&	    27&	        28&	     28         \\
$10^{3}$ & 5893&	    5962&	5924&	     5869       \\
$10^{6}$ & 5881117&	    5950418&	5912190&	     5858016    \\
\hline
\end{tabular}
  }
\end{center}

Let next $\la{W}^{\perp}=\set{w\in\la{W}\,|\,(w, z)=0}$. This choice is motivated
by \eqref{eq:penalty_brezzi} where the inf-sup condition was violated by $z$. In addition,
the solution algorithm for \eqref{eq:elast_stokes_weak} on $\la{W}$ could be designed
following the idea of the Sherman-Morrison-Woodbury formula, that is, by considering
$\mathcal{A}$ on $\la{W}$ as a rank-one perturbation of $\mathcal{A}$ on $\la{W}^{\perp}$
where the latter can be analyzed by Theorem \ref{thm:penalty}.

Using the Riesz map preconditioner based on \eqref{eq:SN_DD_space} Figure
\ref{fig:SN_DD} shows\footnote{The discrete eigenvalue problems were restricted
  to the appropriate subspace by passing to the iterative Krylov-Schur solver
  the interpolant of $z$. 
}
the condition numbers of the preconditioned problem $\mathcal{A}x=\beta\mathcal{B}^{-1}x$ with $x\in\la{W}^{\perp}$.
It can be seen that the values are bounded in all the parameter variations. Robustness of the results then supports the claim that the Brezzi conditions
\eqref{eq:penalty_brezzi} are satisfied on the $\la{W}^{\perp}$. However, we do not prove the claim here.

  \begin{figure}
    \centering
  \includegraphics[height=0.65\textwidth]{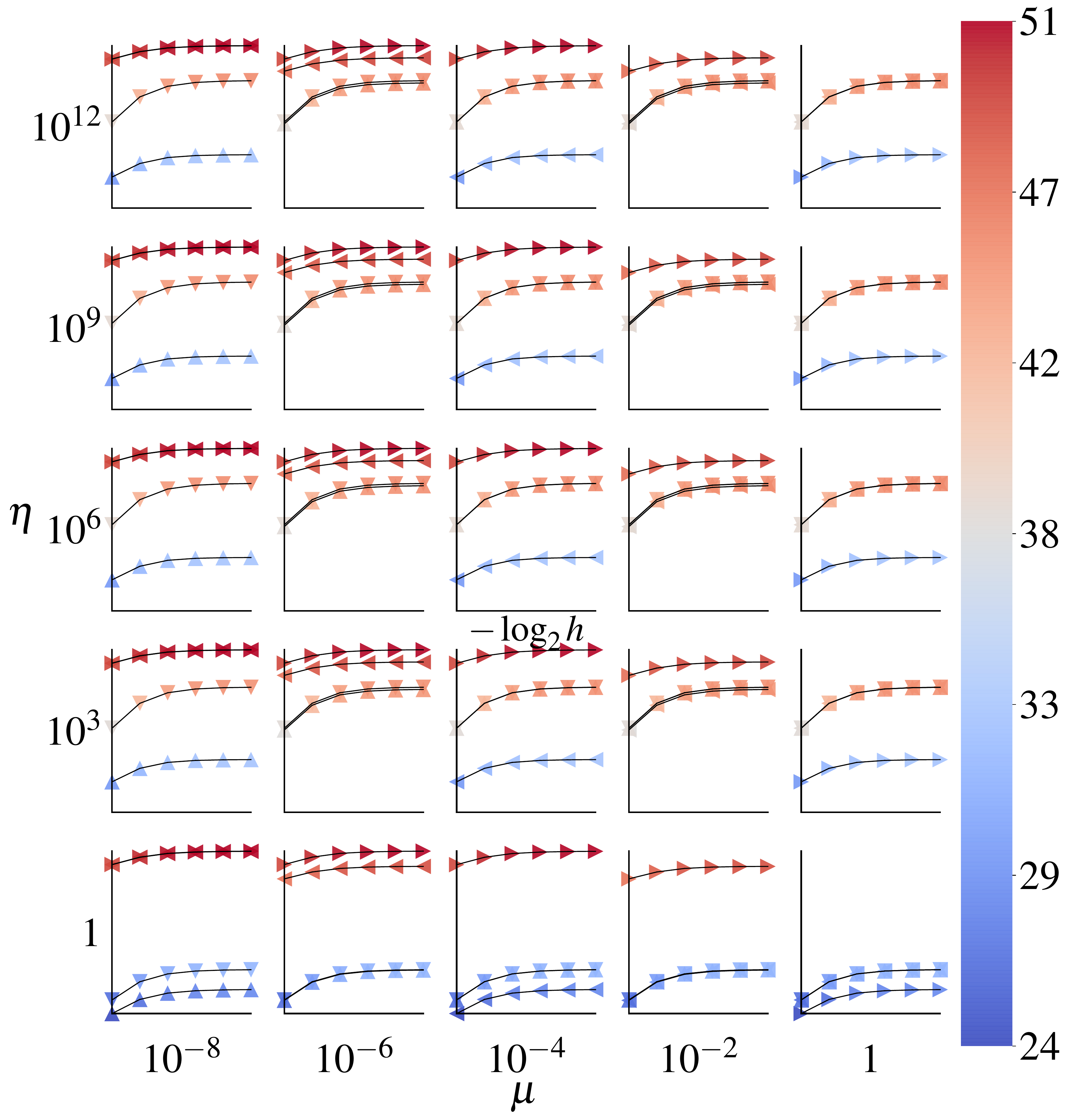}
    \vspace{-15pt}  
   \caption{
     Conditioning of Stokes-Navier problem with homogeneous Dirichlet boundary conditions and preconditioner based on 
     \eqref{eq:SN_DD_space}. Eigenvalue problem is considered 
     on the subspace $\la{W}^\perp$, see Remark \ref{rmrk:SN_DD}.
    For fixed $\mu$, $\eta$ the system size grows from left 
    to right. Values of $k=10^{-6}, 10^{-4}, 10^{-2}, 1$ 
    are encoded with markers $\triangledown$, $\triangle$, $\triangleleft$, $\triangleright$.
   }
  \label{fig:SN_DD}
 \end{figure}
\end{remark}

 Based on the observed bounded spectrum in Remark \ref{rmrk:SN_DD}, iterative solvers for the Stokes-Navier problem with Dirichlet boundary conditions shall be pursued in the future work.

\appendix
\section{Solution times} To allow for comparison of our monolithic approach
with other solution techniques, in particular DD methods,
we list below the solution times of Darcy-Stokes (Example \ref{ex:darcy_right}) and Stokes-Navier (Example \ref{ex:stokes_navier}) problems
preconditioned respectively with \eqref{eq:B_darcy_stokes} and \eqref{eq:B_stokes_navier}
and the timings of subproblems from Examples \ref{ex:stokesreg}, \ref{ex:darcyreg}, \ref{ex:navier}.
With the DD algorithm in mind we also consider the subproblems where
the Dirichlet boundary conditions are enforced by construction of the function space,
i.e. without the Lagrange multiplier.

The experiments are conducted with the setup according to Remark \ref{rmrk:setup} with
all the material parameters set to unity. In particular,
the preconditioners use LU and thus the results present an idealized scenario.
Further, the subproblems are considered on half domain, i.e. $\Omega_f=\left[0, \frac{1}{2}\right]\times\left[0, 1\right]$.
Thus dimension of the discrete space, $\dim \la{W}^h$, in the
coupled problems can be inferred from the corresponding dimensions shown in Table \ref{tab:dims}. 
We remark that the preconditioners for Darcy, Stokes and
Navier subproblems with the standard Dirichlet boundary conditions are defined as
Riesz mappings for $\la{H}_{0, D}(\operatorname{div}, \Omega_f)\times L^2(\Omega_f)$,
$\la{H}_{0, D}^1(\Omega_f)\cap\la{L}^2(\Gamma) \times L^2(\Omega_f)$ and
$\la{H}_{0, D}^1(\Omega_f) \times L^2(\Omega_f)$ respectively.

\begin{table}[h]
\footnotesize{
  \begin{center}
    \begin{tabular}{c|cc|cc|c}
      \hline
    $h$ & $\dim \la{V}^h_{f}$ & $\dim Q^h_{f}$ & $\dim \la{V}^h_{p}$ &$\dim Q^h_{p}$ & $\dim \Lambda^h$\\
      \hline
$2^{-3}$& 1122&    153&    408&      256&    16\\
$2^{-4}$& 4290&    561&    1584&     1024&   32\\
$2^{-5}$& 16770&   2145&   6240&     4096&   64\\
$2^{-6}$& 66306&   8385&   24768&    16384&  128\\
$2^{-7}$& 263682&  33153&  98688&    65536&  256\\
$2^{-8}$& 1051650& 131841& 393984&   262144& 512\\
  \hline
  \end{tabular}
\end{center}  
}
\caption{Dimensions of \DSelm\ finite element spaces used in solver comparison summarized in Table \ref{tab:cpu}.}
\label{tab:dims}
\end{table}

\begin{table}[h]
\footnotesize{
  \begin{center}
    \begin{tabular}{c|c|cc|cc||c|cc}
      \hline
    \multirow{2}{*}{$h$} & \multicolumn{5}{c||}{Darcy-Stokes} & \multicolumn{3}{c}{Stokes-Navier}\\
    \cline{2-9}
    & \eqref{eq:darcy_stokes_weak} &
    \eqref{eq:bab_darcy_op} & \eqref{eq:bab_darcy_op}\textsuperscript{*} &
      \eqref{eq:bab_stokes_op} & \eqref{eq:bab_stokes_op}\textsuperscript{*} &    
      \eqref{eq:elast_stokes_weak} & \eqref{eq:bab_navier_op} & \eqref{eq:bab_navier_op}\textsuperscript{*}\\
      \hline
$2^{-3}$& 0.08&  0.02& $<0.01$& 0.06&  0.03&  0.21&   0.13&  0.04\\ 
$2^{-4}$& 0.16&  0.03& 0.01&    0.12&  0.08&  0.40&   0.24&  0.08\\ 
$2^{-5}$& 1.09&  0.10& 0.02&    0.87&  0.61&  3.05&   1.80&  0.65\\ 
$2^{-6}$& 3.83&  0.57& 0.15&    3.45&  2.62&  10.96&  6.26&  2.54\\ 
$2^{-7}$& 15.15& 2.24& 0.65&    12.67& 9.24&  36.78&  23.54& 9.51\\ 
$2^{-8}$& 44.40& 8.66& 2.55&    39.33& 29.11& 115.62& 74.04& 29.89\\
\hline
\hline      
iter & 50 & 28 & 8 & 61 & 45 & 95 & 116 & 47\\
cond & 6.63 & 3.54 & 1.10 & 21.56 & 6.99 & 20.16 & 55.30 & 12.72\\
  \hline
  \end{tabular}
\end{center}  
}
\caption{Timings of MinRes solver (in seconds, excluding preconditioner setup).
  Asterisk indicates that subproblem does not use Lagrange multiplier and has \emph{all} Dirichlet boundary conditions 
  enforced strongly. Final row shows iteration count till convergence and the condition
  numbers of the preconditioned problems on mesh $h=2^{-8}$.
}
\label{tab:cpu}
\end{table}

In Table \ref{tab:cpu} we report solution times of the MinRes solver running
on a single core of Intel i7 4790S @3.20GHz CPU and with 32GB of memory. Considering
the timings obtained on the finest mesh, it can be seen that for a Darcy-Stokes
problem the cost of a single DD iteration is cca. 47s
if the subproblems are setup using \eqref{eq:bab_darcy_op} and \eqref{eq:bab_stokes_op}.
The cost reduces to cca. 31s if standard Dirichlet conditions are used. For Stokes-Navier
problem the DD iteration take 60 and 150 seconds respectively. The monolithic solution
algorithm presented here thus compares favourably with the domain decomposition approach.
In particular, for similar performance rapid convergence of the (naive) DD in about 2 iterations
is required.

We remark that the condition numbers of the subproblems reported in Table \ref{tab:cpu}
concern $\Omega_f=\left[0, \frac{1}{2}\right]\times\left[0, 1\right]$, while in 
Table \ref{tab:darcyreg}, \ref{tab:stokesreg} and \ref{tab:navier} domain 
$\Omega_f=\left[0, 1\right]^2$ is used. 

\section{Approximation errors} Error convergence of the solutions of the coupled problems 
with unit parameters computed by the monolithic solvers is shown in Figure 
\ref{fig:coupled_cvrg}. We recall that \DSelm\ and \SNelm\ elements were used. 
Convergence rates of the coupled Darcy-Stokes problem are in agreement with 
the estimates established in \cite{galvis2007non}.

\begin{figure}[h]
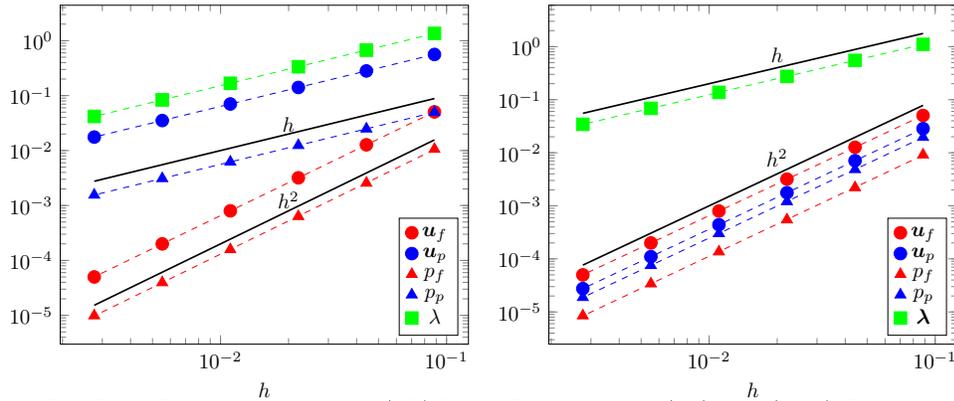

\begin{center}  
\includegraphics[width=0.49\textwidth]{./img/inline_cvrg_darcy_stokes.tex}
\includegraphics[width=0.49\textwidth]{./img/inline_cvrg_stokes_navier.tex}
\end{center}
\vspace{-15pt}
\caption{
  Error convergence for (left) Darcy-Stokes problem \eqref{eq:darcy_stokes_weak}
  and (right) Stokes-Navier problem \eqref{eq:elast_stokes_weak} in the norms
  induced by \eqref{eq:B_darcy_stokes} and \eqref{eq:B_stokes_navier} respectively.
}
\label{fig:coupled_cvrg}
\end{figure}


\bibliographystyle{siamplain}
\bibliography{multiphysics}

\end{document}